\documentclass{amsart} 

\usepackage{amsmath}
\usepackage{amsthm}
\usepackage{amssymb}
\usepackage{color}
\usepackage{graphicx}
\usepackage{mathrsfs}

\newcommand{\dd}{\mathrm{d}}
\newcommand{\ii}{\mathrm{i}}
\newcommand{\ee}{\mathrm{e}}
\newcommand{\arctanh}{\mathop{\mathrm{arctanh}}}
\newcommand{\psum}{\sideset{_{}^{}}{_{}^{\prime}}\sum}
\newcommand{\pprod}{\sideset{_{}^{}}{_{}^{\prime}}\prod}

\newtheorem{thm}{Theorem}[section]
\newtheorem{lem}[thm]{Lemma}

\theoremstyle{definition}
\newtheorem{defn}[thm]{Definition}

\theoremstyle{remark}
\newtheorem{rem}[thm]{Remark}

\allowdisplaybreaks[4]

\numberwithin{equation}{section}

\makeatletter
\newcommand{\figcaption}[1]{\def\@captype{figure}\caption{#1}}
\newcommand{\tblcaption}[1]{\def\@captype{table}\caption{#1}}
\makeatother

\begin{document}

\title[An optimal approximation formula]
{An optimal approximation formula for functions with singularities}


\author{Ken'ichiro Tanaka}
\address{Department of Mathematical Engineering, Faculty of Engineering, Musashino University,
3-3-3, Ariake, Koto-ku, Tokyo 135-8181, Japan}
\curraddr{}
\email{ketanaka@musashino-u.ac.jp}
\thanks{}

\author{Tomoaki Okayama}
\address{Department of Systems Engineering, Graduate School of Information Sciences, Hiroshima City University, 
3-4-1, Ozuka-higashi, Asaminami-ku, Hiroshima 731-3194, Japan}
\curraddr{}
\email{okayama@hiroshima-cu.ac.jp}
\thanks{}

\author{Masaaki Sugihara}
\address{Department of Physics and Mathematics, College of Science and Engineering, Aoyama Gakuin University
5-10-1, Fuchinobe, Chuo-ku, Sagamihara-shi, Kanagawa 252-5258, Japan}
\curraddr{}
\email{sugihara@gem.aoyama.ac.jp}
\thanks{}

\subjclass[2010]{Primary 65D15; Secondary 41A25}

\keywords{Hardy space, endpoint singularity, optimal approximation, Ganelius sampling points}

\date{October 22, 2016}

\dedicatory{}

\begin{abstract}
We propose an optimal approximation formula for analytic functions 
that are defined on a complex region containing the real interval $(-1,1)$ and 
possibly have algebraic singularities at the endpoints of the interval. 
As a space of such functions, 
we consider a Hardy space with the weight given by $w_{\mu}(z) = (1-z^{2})^{\mu/2}$ for $\mu > 0$, 
and formulate the optimality of an approximation formula for the functions in the space. 
Then, we propose an optimal approximation formula for the space for any $\mu > 0$ 
as opposed to existing results with the restriction $0 < \mu < \mu_{\ast}$ for a certain constant $\mu_{\ast}$. 
We also provide the results of numerical experiments to show the performance of the proposed formula. 
 
\end{abstract}

\maketitle

\section{Introduction}
\label{sec:intro}
This paper is concerned with approximation of functions 
by a finite number of the sampled values of them. 
We consider analytic functions that are defined on a complex region containing a real interval and 
possibly have endpoint singularities on the interval. 
In order to deal with such functions collectively,  
we consider a function space consisting of them
and formulate the optimality of an approximation formula for the functions in the space. 
Then, we propose an optimal approximation formula for the function space. 

We consider the region given by 
\begin{align}
\notag
\varLambda_d = 
\left\{
z \in \mathbf{C} 
\left| \ 
\left|
\mathop{\mathrm{arg}} \left( \frac{1 + z}{1 - z} \right)
\right| < d
\right.
\right\}, 
\end{align}
which satisfies $\varLambda_d \cap \mathbf{R} = (-1, 1)$.
In order to deal with analytic functions on $\varLambda_d$ with 
algebraic singularities at the endpoints $\pm 1$, 
we consider the function space given by
\begin{align}
\notag
\boldsymbol{H}^{\infty}(\varLambda_{d}, w_{\mu}) 
=
\left\{ 
f: \varLambda_{d} \to \mathbf{C} 
\left| \,  
\text{$f$ is analytic in $\varLambda_{d}$ and } 
\sup_{z \in \varLambda_{d}} \left| \frac{f(z)}{w_{\mu}(z)} \right| < \infty  
\right.
\right\}, 
\end{align}
where $\mu$ is a positive number and $w_{\mu}(z) = (1-z^{2})^{\mu/2}$. 
This space has been studied as a fundamental space for 
the sinc numerical methods~\cite{bib:StengerBook1993, bib:StengerBook2011}, 
which are the numerical methods based on the approximation of functions by the sinc function (see~\eqref{eq:SE-Sinc}). 
The error analysis of the sinc approximation has been performed 
in these decades~\cite{bib:StengerBook1993, 
bib:StengerSincSummary2000, 
bib:StengerBook2011, 
TanaSugiMuro_DE_Sinc_2009}.
It is well-known that the sinc approximation has very good accuracy in $\boldsymbol{H}^{\infty}(\varLambda_{d}, w_{\mu})$. 

Besides the studies of such concrete formulas in $\boldsymbol{H}^{\infty}(\varLambda_{d}, w_{\mu})$, 
there are several analyses of the errors of optimal formulas 
in spaces of analytic functions like $\boldsymbol{H}^{\infty}(\varLambda_{d}, w_{\mu})$. 
In the literatures~\cite{bib:BurchardHollig_nwidth_1985, bib:StengerMinNorm1978, bib:Sugihara2003_OptSinc, bib:Wilderotter_nwidth_1992}, 
the authors have estimated the optimal errors in Hardy spaces with preassigned decay rates. 
In particular, 
Sugihara~\cite{bib:Sugihara2003_OptSinc} 
has given a lower bound of the optimal error in $\boldsymbol{H}^{\infty}(\varLambda_{d}, w_{\mu})$
and revealed that the sinc approximation is near optimal in the space. 
In order to formulate the optimality of an approximation formula for the functions in $\boldsymbol{H}^{\infty}(\varLambda_{d}, w_{\mu})$, 
he considered all the possible $n$-point approximation formulas in the space and the norms of their error operators. 
Then, he defined the minimum error norm 
$E_{n}^{\mathrm{min}}(\boldsymbol{H}^{\infty}(\varLambda_{d}, w_{\mu}))$ 
by the minimum of the norms. 
Furthermore, he also considered the error norm of the sinc approximation on $\boldsymbol{H}^{\infty}(\varLambda_{d}, w_{\mu})$, 
denoted by $E_{n}^{\mathrm{sinc}}(\boldsymbol{H}^{\infty}(\varLambda_{d}, w_{\mu}))$, 
and has shown that 
\[
c'' \exp(-c_{2} \sqrt{n})
\leq
E_{n}^{\mathrm{min}}(\boldsymbol{H}^{\infty}(\varLambda_{d}, w_{\mu}))
\leq 
E_{n}^{\mathrm{sinc}}(\boldsymbol{H}^{\infty}(\varLambda_{d}, w_{\mu}))
\leq 
c' \sqrt{n} \exp(-c_{1} \sqrt{n}), 
\] 
where $c'$, $c''$, $c_{1}$, and $c_{2}$ are positive constants with $c_{1} < c_{2}$ (see \eqref{eq:NearOptSinc}). 

However, 
finding an explicit approximation formula attaining 
$E_{n}^{\mathrm{min}}(\boldsymbol{H}^{\infty}(\varLambda_{d}, w_{\mu}))$
has been an open problem so far 
whereas the exact order of $E_{n}^{\mathrm{min}}(\boldsymbol{H}^{\infty}(\varLambda_{d}, w_{\mu}))$
with respect to $n$ is known in some restricted case. 
Recently, in the restricted case that $0 < \mu < \min\{ 2, \pi/d \}$, 
Ushima et al.~\cite{bib:UTOS_Ganelius_JSIAM} have proposed an optimal formula 
by using the technique of Jang and Haber~\cite{bib:JangHaber2001}, 
in which they employ a modification of the sampling points given by Ganelius~\cite{bib:Ganelius_1976}.  
The restriction $0 < \mu < \min \{2, \pi/d \}$ is owing to the assumption $r < 1$ in the Ganelius theorem~\cite[Lemma~1]{bib:JangHaber2001}, 
which plays an important role for the error estimate of the formula. 
In this paper, we remove this restriction and propose an optimal formula for any $\mu > 0$
by generalizing the formula in~\cite{bib:UTOS_Ganelius_JSIAM}. 

The rest of this paper is organized as follows. 
In Section~\ref{sec:math_pre}, 
we list mathematical tools for setting the framework for approximation of the functions in $\boldsymbol{H}^{\infty}(\varLambda_{d}, w_{\mu})$. 
We give the more precise explanations of the region $\varLambda_{d}$, space $\boldsymbol{H}^{\infty}(\varLambda_{d}, w_{\mu})$, 
and the notion of the optimal approximation in $\boldsymbol{H}^{\infty}(\varLambda_{d}, w_{\mu})$. 
Furthermore, we review some existing results about the estimate of 
$E_{n}^{\mathrm{min}}(\boldsymbol{H}^{\infty}(\varLambda_{d}, w_{\mu}))$. 
In Section~\ref{sec:approx_formula}, 
we present our new formula and show its error estimate in Theorem~\ref{thm:general-Ganelius-approximation}. 
By combining this theorem and the existing result giving the lower estimate of 
$E_{n}^{\mathrm{min}}(\boldsymbol{H}^{\infty}(\varLambda_{d}, w_{\mu}))$, 
we show the optimality of the proposed formula. 
The proof of Theorem~\ref{thm:general-Ganelius-approximation} is owing to three lemmas, 
whose proofs are presented in Section~\ref{sec:proof_of_main_steps}. 
The last one of them,  Lemma~\ref{lem:from_gen_Ganelius}, is proven by being reduced to Theorem~\ref{thm:GenGan}, 
a generalization of the Ganelius theorem without the assumption $r < 1$. 
This theorem is proven in Appendix~\ref{sec:proof_of_generalGan}. 
In Section~\ref{sec:num_exp}, 
we present some numerical results showing the performance of our formula. 
Finally, we conclude this paper in Section~\ref{sec:concl}.

\section{Mathematical preliminaries and existing results}
\label{sec:math_pre}
\subsection{Function space $\boldsymbol{H}^{\infty}(\varLambda_{d}, w_{\mu})$}

For a real number $d$ with $0 < d < \pi$, 
we consider the strip region $\mathcal{D}_{d} := \{ \zeta \in \mathbf{C} \mid | \mathop{\mathrm{Im}} \zeta \, | < d \}$. 
Then, we define a region $\varLambda_{d}$ by
\begin{align}
\label{eq:def_Lambda_d}
\varLambda_d := 
\left\{ 
z \in \mathbf{C} 
\mid
z = \tanh(\zeta/2), \ \zeta \in \mathcal{D}_{d}
\right\}, 
\end{align}
and set the counterclockwise direction to its boundary $\partial \varLambda_{d}$. 
The region $\varLambda_d$ can be written in the form 
\begin{align}
\notag
\varLambda_d = 
\left\{
z \in \mathbf{C} 
\left| \ 
\left|
\mathop{\mathrm{arg}} \left( \frac{1 + z}{1 - z} \right)
\right| < d
\right.
\right\}. 
\end{align}
Furthermore, the region $\varLambda_d$ is symmetric with respect to the real axis and satisfies $\varLambda_d \cap \mathbf{R} = (-1, 1)$. 
The upper half part of $\varLambda_d$ coincides with the intersection of the upper half plane and the open disc with 
center $(-\mathrm{i}/\tan d)$ and radius $1/\sin d$. 
The intersection of the boundary $\partial \varLambda_{d}$ and the imaginary axis 
consists of $(\pm \mathrm{i} \tan (d/2))$. 
In particular, 
the region $\varLambda_d$ becomes 
an eye-shaped region if $0 < d < \pi/2$, 
a unit disc if $d = \pi/2$, and 
the entire complex plane if $d \to \pi$.
As an example, we show the region $\varLambda_{\pi/4}$ in Figure~\ref{fig:region}.

\begin{figure}[!ht]
\centering
\includegraphics[width=.6\linewidth]{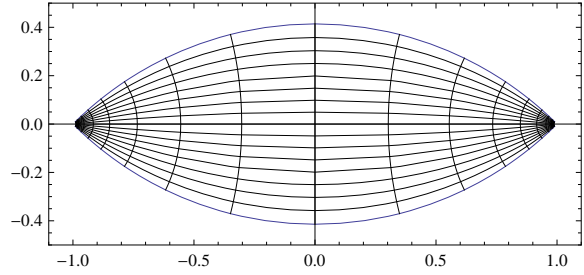}
\caption{Region $\varLambda_{\pi/4}$.}
\label{fig:region}
\end{figure}

Throughout this paper, 
we consider approximation of analytic functions defined on $\varLambda_{d}$
that possibly have algebraic singularities at the endpoints $\pm 1$. 
Accordingly, we introduce a function space of such functions as below. 

\begin{defn}
\label{defn:Hardy}
Let $\mu$ be a positive number and let 
\begin{align}
\notag
w_{\mu}(z) := (1-z^{2})^{\mu/2}. 
\end{align}
We define a function space $\boldsymbol{H}^{\infty}(\varLambda_{d}, w_{\mu})$ by 
\begin{align}
\notag
\boldsymbol{H}^{\infty}(\varLambda_{d}, w_{\mu}) 
:=
\{ f: \varLambda_{d} \to \mathbf{C} \mid \text{$f$ is analytic in $\varLambda_{d}$ and } \| f \| < \infty  \}, 
\end{align}
where 
\begin{align}
\notag
\| f \| := \sup_{z \in \varLambda_{d}} \left| \frac{f(z)}{w_{\mu}(z)} \right|. 
\end{align} 
\end{defn}

\begin{rem}
\label{rem:Hardy}
A function $f \in \boldsymbol{H}^{\infty}(\varLambda_{d}, w_{\mu})$ satisfies 
\begin{align}
\notag
| f(z) | \leq \| f \| \, | (1 - z^{2})^{\mu/2} | 
\end{align}
for any $z \in \varLambda_{d}$. 
Therefore, the function $f$ tends to zero with order $\mathrm{O}((1\pm z)^{\mu/2})$ as $z \to \mp 1$. 
\end{rem}

\subsection{Minimum error norm $E_{n}^{\mathrm{min}}(\boldsymbol{H}^{\infty}(\varLambda_{d}, w_{\mu}))$}

For $f \in \boldsymbol{H}^{\infty}(\varLambda_{d}, w_{\mu})$, 
we consider all the possible $n$-point approximation formulas written in the form 
\begin{align}
\label{eq:approx_general}
f(x) \approx \sum_{j=1}^{\ell} \sum_{k=0}^{m_{j}-1} f^{(k)}(a_{j})\, \phi_{jk}(x), 
\end{align}
where 
$\ell$ is an integer with $1\leq \ell \leq n$, 
$\{ m_{j} \}$ is a sequence of nonnegative integers with $m_{1} + \cdots + m_{\ell} = n$, 
$\{ a_{j} \}$ is a sequence of sampling points in $\varLambda_{d}$, and 
$\{ \phi_{jk} \}$ is a sequence of analytic functions on $\varLambda_{d}$. 
Then, let $\mathcal{N}^{\ell, m_{i}}_{a_{j}, \phi_{jk}}$ denote the operator norm  
of the error operator associated with Formula~\eqref{eq:approx_general}:
\begin{align}
\label{eq:norm_err_op_general}
\mathcal{N}^{\ell, m_{i}}_{a_{j}, \phi_{jk}}
:= 
\sup_{
\begin{subarray}{c}
f \in \boldsymbol{H}^{\infty}(\varLambda_{d}, w_{\mu}) \\
\| f \| \leq 1
\end{subarray}}
\left[
\sup_{x \in (-1,1)} 
\left| f(x) - \sum_{j=1}^{\ell} \sum_{k=0}^{m_{j}-1} f^{(k)}(a_{j})\, \phi_{jk}(x) \right|
\right].
\end{align}
We call the value $\mathcal{N}^{\ell, m_{i}}_{a_{j}, \phi_{jk}}$ the error norm of Formula~\eqref{eq:approx_general}
and adopt it as a criterion for evaluating the accuracy of Formula~\eqref{eq:approx_general}. 
Therefore, an approximation formula with the form in~\eqref{eq:approx_general} is \emph{optimal} 
if it achieves the infimum of the error norm $\mathcal{N}^{\ell, m_{i}}_{a_{j}, \phi_{jk}}$ over any $n$-point approximation formulas. 
By letting the infimum denoted by 
\begin{align}
\label{eq:min_err_norm}
E_{n}^{\mathrm{min}}(\boldsymbol{H}^{\infty}(\varLambda_{d}, w_{\mu}))
:=
\inf_{1\leq \ell \leq n} 
\inf_{
\begin{subarray}{c}
m_{1}, \ldots, m_{\ell} \\
m_{1} + \cdots + m_{\ell} = n
\end{subarray}
}
\inf_{a_{j} \in \varLambda_{d}} 
\inf_{\phi_{jk}}
\mathcal{N}^{\ell, m_{i}}_{a_{j}, \phi_{jk}}, 
\end{align}
we call it the minimum error norm of the $n$-point approximation in $\boldsymbol{H}^{\infty}(\varLambda_{d}, w_{\mu})$. 
In the literature 
\cite{bib:Sugihara2003_OptSinc}, 
a lower bound of 
$E_{n}^{\mathrm{min}}(\boldsymbol{H}^{\infty}(\varLambda_{d}, w_{\mu}))$ 
is given. 

\begin{thm}[{\cite[(a) on page 782]{bib:Sugihara2003_OptSinc}}]
\label{thm:min_err_lower_bound}
The minimum error norm \eqref{eq:min_err_norm} is bounded from below as follows:
\begin{align}
\label{eq:min_err_lower_bound}
E_{n}^{\mathrm{min}}(\boldsymbol{H}^{\infty}(\varLambda_{d}, w_{\mu})) 
\geq 
c\, \exp \left( -\sqrt{\pi d \mu n/2} \right), 
\end{align}
where $c$ is a positive number independent of $n$. 
\end{thm}

On the other hand, 
an upper bound of $E_{n}^{\mathrm{min}}(\boldsymbol{H}^{\infty}(\varLambda_{d}, w_{\mu}))$
can be given by the error norm $\mathcal{N}^{\ell, m_{i}}_{a_{j}, \phi_{jk}}$ of 
an approximation formula that is applicable to the functions in $\boldsymbol{H}^{\infty}(\varLambda_{d}, w_{\mu})$. 
It is well-known that 
an upper bound close to the lower bound in~\eqref{eq:min_err_lower_bound} 
is given by the sinc approximation formula with a variable transformation of a single exponential type 
as shown in the next subsection. 

\subsection{Nearly optimal formula (Sinc approximation)}

The $(2N+1)$-point sinc approximation is defined by 
\begin{align}
\label{eq:SE-Sinc}
f(x) \approx
\sum_{j = -N}^{N} f(\psi(jh)) \, S(j, h)(\psi^{-1}(x)), 
\end{align}
where 
\begin{align}
\notag
h &:= \sqrt{\frac{2 \pi d}{\mu N}}, \\
\notag
\psi(\zeta) &:= \tanh(\zeta/2),
\end{align}
and
\begin{align}
\notag
S(j, h)(t) := \frac{\sin [\pi (t/h - j)]}{\pi (t/h - j)}. 
\end{align}
Formula~\eqref{eq:SE-Sinc} is called the SE-Sinc formula, 
which has been intensively studied by Stenger et al.~\cite{
bib:LundBowers1992, 
bib:StengerBook1993, 
bib:StengerSincSummary2000, 
bib:StengerBook2011}.
Let $E_{2N+1}^{\mathrm{sinc}}(\boldsymbol{H}^{\infty}(\varLambda_{d}, w_{\mu}))$
be the error norm $\mathcal{N}^{\ell, m_{i}}_{a_{j}, \phi_{jk}}$ of Formula~\eqref{eq:SE-Sinc}, i.e., 
\begin{align}
& E_{2N+1}^{\mathrm{sinc}}(\boldsymbol{H}^{\infty}(\varLambda_{d}, w_{\mu})) \\
& :=
\sup_{
\begin{subarray}{c}
f \in \boldsymbol{H}^{\infty}(\varLambda_{d}, w_{\mu}) \\
\| f \| \leq 1
\end{subarray}}
\left[
\sup_{x \in (-1,1)} 
\left| f(x) - \sum_{j = -N}^{N} f(\psi(jh)) \ S(j, h)(\psi^{-1}(x)) \right|
\right]. \notag
\end{align}
Sugihara~\cite{bib:Sugihara2003_OptSinc} has shown the following upper bound of 
$E_{2N+1}^{\mathrm{sinc}}(\boldsymbol{H}^{\infty}(\varLambda_{d}, w_{\mu}))$. 

\begin{thm}[{\cite[(a) on page 782]{bib:Sugihara2003_OptSinc}}]
\label{thm:sinc_err_upper_bound}
The minimum error norm of the SE-Sinc formula is bounded from above as follows:
\begin{align}
E_{2N+1}^{\mathrm{sinc}}(\boldsymbol{H}^{\infty}(\varLambda_{d}, w_{\mu})) 
\leq 
c'\, \sqrt{N} \exp\left( -\sqrt{\pi d \mu N / 2} \right),
\end{align}
where $c'$ is a positive number independent of $N$.
\end{thm}

From Theorem~\ref{thm:min_err_lower_bound} with $n = 2N+1$ and Theorem~\ref{thm:sinc_err_upper_bound}, 
we have
\begin{align}
c''\, \exp \left( -\sqrt{\pi d \mu N} \right)
& \leq 
E_{2N+1}^{\mathrm{min}}(\boldsymbol{H}^{\infty}(\varLambda_{d}, w_{\mu})) \label{eq:NearOptSinc} \\
& \leq 
E_{2N+1}^{\mathrm{sinc}}(\boldsymbol{H}^{\infty}(\varLambda_{d}, w_{\mu}))
\leq 
c'\, \sqrt{N} \exp\left( -\sqrt{\pi d \mu N / 2} \right) \notag
\end{align}
for some positive numbers $c''$ and $c'$ independent of $N$, 
which gives an estimate of the order of the minimum error norm 
$E_{2N+1}^{\mathrm{min}}(\boldsymbol{H}^{\infty}(\varLambda_{d}, w_{\mu}))$
with respect to $N$. 

Recently, 
the exact order of the minimum error norm 
is revealed by an explicit approximation formula 
as shown in the next subsection. 


\subsection{Optimal formula}

Ushima et al.~\cite{bib:UTOS_Ganelius_JSIAM} have found out an explicit approximation formula 
that achieves the exact order of the minimum error norm 
for $\mu$ with $0 < \mu < \min \{ 2, \pi/d \}$ by using the modified Ganelius sampling points proposed in~{\cite[Lemma 1]{bib:JangHaber2001}}. 
Furthermore, they have shown that
\begin{align}
\label{eq:exact_min_err_norm}
E_{2N}^{\mathrm{min}}(\boldsymbol{H}^{\infty}(\varLambda_{d}, w_{\mu}))
\leq
C \exp \left( -\sqrt{\pi d \mu N} \right)
\end{align}
for a positive number $C$ independent of $N$, and that the RHS in \eqref{eq:exact_min_err_norm}
gives the exact order of $E_{2N}^{\mathrm{min}}(\boldsymbol{H}^{\infty}(\varLambda_{d}, w_{\mu}))$
by combining this inequality and Theorem~\ref{thm:min_err_lower_bound} with $n = 2N$. 
In order to show the proposed formula in~\cite{bib:UTOS_Ganelius_JSIAM}, 
we describe the definition of the modified Ganelius sampling points
and a fundamental inequality relating to them, 
which plays an important role for the error estimate of the formula. 

\begin{defn}
\label{defn:Ganelius_points}
Let $r$ be a positive real number and let $N$ be a positive integer. 
Furthermore, let $N_{0}$ be defined by 
\begin{align}
\notag
N_{0} := N - \left \lceil \frac{\pi}{4} \sqrt{N r} \right \rceil, 
\end{align}
and let $\varphi$ be the function defined by 
\begin{align}
\notag
\varphi(x) := \exp \left( \pi \sqrt{\frac{x}{r}} \right)
\end{align}
for a positive number $x$. 
Then, the numbers $a_{k}$ defined by 
\begin{align}
\notag
a_{k} :=
\begin{cases}
\varphi(k-1)/\varphi(N_{0}) & (k = 1,2,\ldots, N_{0}), \\
\varphi(k-3/2)/\varphi(N_{0}) & (k = N_{0}+1), \\
1 - \frac{k-N_{0}-1}{5(N-N_{0}-1)} & (k = N_{0}+2, \ldots, N)
\end{cases}
\end{align}
are called the modified Ganelius sampling points.
\end{defn}

\begin{thm}[{\cite[Lemma 1]{bib:JangHaber2001}}]
\label{thm:OldGan}
Let $r$ be a positive real number satisfying $r < 1$ and let $N$ be a positive integer. 
Furthermore, let $\{ a_{k} \}$ be the sequence of the modified Ganelius sampling points given by Definition~\ref{defn:Ganelius_points}. Then, 
\begin{align}
\max_{s \in [0,1]} s^{r} \prod_{k = 1}^{N} \left| \frac{s-a_{k}}{s+a_{k}} \right| \leq C \exp \left( - \pi \sqrt{N r} \right)
\label{eq:old_mod_Gan}
\end{align}
holds, where $C$ is a positive number independent of $N$. 
\end{thm}

The sequence $\{ a_{k} \}$ given by Definition~\ref{defn:Ganelius_points} 
is contained in the interval $(0,1)$. 
We need to transform this sequence to that on the interval $(-1,1)$ for the approximation formula on $(-1,1)$. 
Let $b_{k}$ be defined by
\begin{align}
\label{eq:b_k}
b_{k} := \sqrt{\frac{1 - a_{k}}{1 + a_{k}}}, \qquad 
b_{-k} := -b_{k} \qquad (k = 1,2,\ldots, N)
\end{align}
and let $\beta_{k}$ be defined by 
\begin{align}
\label{eq:beta_k}
\beta_{k} := \tanh \left( \frac{2 d}{\pi} \arctanh b_{k} \right)
\qquad (k = \pm 1, \ldots, \pm N). 
\end{align}
Both of the sequences $\{ b_{k} \}$ and $\{ \beta_{k} \}$ are contained in $(-1,1)$. 
Furthermore, 
we define $\sigma_{k}$ by
\begin{align}
\sigma_{k} :=
\pprod_{\begin{subarray}{c}
\ell=-N \\
\ell \neq k
\end{subarray}}^{N} \frac{1 - b_{\ell}b_{k}}{b_{k} - b_{\ell}}
\end{align}
for the coefficients of the formula,
where the symbol $'$ of the product symbol means exclusion of $k = 0$. 
We also use the same symbol for the summation symbol. 
Finally, in order to construct basis functions for the formula, 
we define a function $B_{N}(z; \boldsymbol{\beta}, d)$ by
\begin{align}
\label{eq:GenBlaschke}
B_{N}(z; \boldsymbol{\beta}, d)
:=
\pprod_{k=-N}^{N} \tanh \left[ \frac{\pi}{2d} (\arctanh z - \arctanh \beta_{k}) \right].
\end{align}

\begin{rem}
\label{rem:special_case_Blaschke}
After some algebra, 
we can obtain the expression 
\begin{align}
\notag
B_{N}(z; \boldsymbol{\beta}, d)
=
\pprod_{k=-N}^{N} 
\frac{ [(1-\beta_{k})(1+z)]^{\pi/(2d)} - [(1+\beta_{k})(1-z)]^{\pi/(2d)} }
{ [(1-\beta_{k})(1+z)]^{\pi/(2d)} + [(1+\beta_{k})(1-z)]^{\pi/(2d)} }. 
\end{align}
Therefore, if $d = \pi/(2m)$ for a positive integer $m$, 
the function $B_{N}(z; \boldsymbol{\beta}, d)$ is a rational function. 
In particular, if $d = \pi/2$, we have $\beta_{k} = b_{k}$ and 
\begin{align}
\notag
B_{N}(z; \boldsymbol{\beta}, \pi/2)
=
\pprod_{k=-N}^{N} \frac{z - b_{k}}{1 - b_{k} z},
\end{align}
which is known as the Blaschke product. 
Therefore, the function given by~\eqref{eq:GenBlaschke} is its generalization. 
\end{rem}

By using the sequences and function defined above, 
Ushima et al.~\cite{bib:UTOS_Ganelius_JSIAM} have proposed the approximation formula $\tilde{f}_{N}$ given by
\begin{align}
\label{eq:old_formula}
f(x) \approx 
\tilde{f}_{N}(x) := 
\psum_{k = -N}^{N}
f(\beta_{k}) \, 
\frac{2d \sigma_{k}}{\pi} \, 
\frac{(1-x^{2})\, B_{N}(x; \boldsymbol{\beta},  d)}{x - \beta_{k}}.
\end{align}
Then, by using Theorem~\ref{thm:OldGan}, 
they have given its error estimate as follows.

\begin{thm}[{\cite{bib:UTOS_Ganelius_JSIAM}}]
\label{thm:UTOS_Ganelius_JSIAM}
Let $\mu$ be a real number satisfying $0 < \mu < \min \{ 2, \pi/d \}$. 
Then, for $\tilde{f}_{N}$ given by~\eqref{eq:old_formula}, we have 
\begin{align}
\sup_{
\begin{subarray}{c}
f \in \boldsymbol{H}^{\infty}(\mathcal{D}_{d}, w_{\mu}) \\
\| f \| \leq 1
\end{subarray}}
\left(
\sup_{x \in (-1,1)}
\left|
f(x) - \tilde{f}_{N}(x)
\right|
\right)
\leq
C \exp\left(
- \sqrt{\pi d \mu N}
\right),
\end{align}
where $C$ is a positive number independent of $N$. 
\end{thm}

\subsection{Contribution of this paper}

In Theorem~\ref{thm:UTOS_Ganelius_JSIAM}, 
the assumption $\mu < \min \{ 2, \pi/d \}$ is originated from 
the form of Formula~\eqref{eq:old_formula} and the assumption $r < 1$ in Theorem~\ref{thm:OldGan}. 
In this paper, 
by generalizing Formula~\eqref{eq:old_formula}, 
we propose a new approximation formula in $\boldsymbol{H}^{\infty}(\mathcal{D}_{d}, w_{\mu})$ for any $\mu > 0$
and generalize Theorem~\ref{thm:OldGan} by removing the assumption $r < 1$. 
Then, 
we give the error estimate of the new formula 
by a generalization of Theorem~\ref{thm:UTOS_Ganelius_JSIAM}, 
in which the assumption $\mu < \min \{ 2, \pi/d \}$ is removed. 
As shown below, 
the generalized versions of 
Formula~\eqref{eq:old_formula}, 
Theorem~\ref{thm:OldGan}, and
Theorem~\ref{thm:UTOS_Ganelius_JSIAM} 
are 
Formula~\eqref{eq:main_formula}, 
Theorem~\ref{thm:GenGan}, and 
Theorem~\ref{thm:general-Ganelius-approximation}, respectively.

\section{An approximation formula by the Ganelius sampling points and generalized Blaschke product}
\label{sec:approx_formula}
\subsection{Main result}

For a function $f \in \boldsymbol{H}^{\infty}(\varLambda_{d}, w_{\mu})$, 
by using a real number $\nu$ with 
\begin{align}
\mu/2 < \nu < \mu/2 + 1,
\label{eq:range_of_nu}
\end{align} 
we propose the approximation formula $\tilde{f}_{\nu, N}(x)$ given by
\begin{align}
\label{eq:main_formula}
f(x) \approx
\tilde{f}_{\nu, N}(x) := 
\psum_{k = -N}^{N}
f(\beta_{k}) \, 
\frac{2d \sigma_{k}}{\pi} \, 
\frac{(1-x^{2})^{\nu}}{(1-\beta_{k}^{2})^{\nu - 1}} \, 
\frac{B_{N}(x; \boldsymbol{\beta},  d)}{x - \beta_{k}}. 
\end{align}

\begin{rem}
In the case that $\mu < 2$, by choosing $\nu = 1$, 
we can obtain Formula~\eqref{eq:old_formula} from Formula~\eqref{eq:main_formula}. 
\end{rem}

\begin{rem}
According to~\eqref{eq:range_of_nu}, 
we can set $\nu = \lceil \mu/2 \rceil$ if $\mu$ is not an even integer.  
From this fact and Remark~\ref{rem:special_case_Blaschke}, 
in the case that $\mu$ is not an even integer and $d = \pi/(2m)$ for a positive integer $m$, 
the approximant $\tilde{f}_{\nu, N}(x)$ becomes a rational function by letting $\nu = \lceil \mu/2 \rceil$. 
In such a case, it may be better to use the rational approximant from a practical point of view.
\end{rem}

In the following, 
we give an upper bound of the error of Formula~\eqref{eq:main_formula} and 
show its optimality by the fact 
that the upper bound has the same order as the lower bound of  
$E_{2N}^{\mathrm{min}}(\boldsymbol{H}^{\infty}(\varLambda_{d}, w_{\mu}))$ 
given by~\eqref{eq:min_err_lower_bound}. 
The upper bound is given by the following theorem. 

\begin{thm}
\label{thm:general-Ganelius-approximation}
Let $d$ be a positive number satisfying $0 < d < \pi$,
and let $\mu$ and $\nu$ be positive numbers satisfying~\eqref{eq:range_of_nu}. 
Then, for $\tilde{f}_{\nu, N}$ given by~\eqref{eq:main_formula}, we have
\begin{align}
\sup_{
\begin{subarray}{c}
f \in \boldsymbol{H}^{\infty}(\varLambda_{d}, w_{\mu}) \\
\| f \| \leq 1
\end{subarray}
}
\left(
\sup_{x \in (-1,1)}
\left|
f(x) - \tilde{f}_{\nu, N}(x)
\right|
\right)
\leq
C \exp \left( - \sqrt{\pi d \mu N} \right), 
\label{eq:general-Ganelius-error}
\end{align}
where $C$ is a positive number independent of $N$. 
\end{thm}

\subsection{Sketch of the proof of Theorem~\ref{thm:general-Ganelius-approximation}}

Theorem~\ref{thm:general-Ganelius-approximation} 
follows from Lemmas~\ref{lem:UB_step1}--\ref{lem:from_gen_Ganelius} below, 
whose proofs are shown in Section~\ref{sec:proof_of_main_steps}. 
In order to state the lemmas, 
for a nonnegative real number $\delta$, 
we define $\varLambda_d(\delta)$ by
\begin{align}
\varLambda_d(\delta) := \varLambda_{d} \cap \left\{ z \left| \inf_{\zeta \in \partial \Lambda_{d}} |z - \zeta| > \delta \right. \right\}, 
\end{align}
and set the counterclockwise direction to its boundary $\partial\varLambda_d(\delta)$. 

\begin{lem}
\label{lem:UB_step1}
We have
\begin{align}
\label{eq:UB_step1}
& \sup_{
\begin{subarray}{c}
f \in \boldsymbol{H}^{\infty}(\varLambda_{d}, w_{\mu}) \\
\| f \| \leq 1
\end{subarray}
}
\left(
\sup_{x \in (-1,1)}
\left|
f(x) - \tilde{f}_{\nu, N}(x)
\right|
\right) \\
& \leq 
\sup_{x \in (-1,1)} \frac{1}{2\pi} \, (1-x^2)^{\nu} \, | B_n (x ; \boldsymbol{\beta}, d) | 
\, \lim_{\delta \to +0}
\oint_{\partial \varLambda_d(\delta)} \left | \frac{(1-z^2)^{\mu/2-\nu}}{z-x} \right | |\mathrm{d}z|.
\notag
\end{align}
\end{lem}

\begin{lem}
\label{lem:UB_step2}
Let $x$ be a real number with $x \in (-1,1)$. 
Then, we have
\begin{equation}
\lim_{\delta \to +0}
\oint_{\partial \varLambda_d(\delta)} 
\left | \frac{(1-z^2)^{\mu/2-\nu}}{z-x} \right | |\mathrm{d}z| \leq C_{1} (1-x^{2})^{\mu/2 - \nu}. 
\label{eq:general-Const}
\end{equation}
where $C_{1}$ is a positive real number independent of $x$.
\end{lem}

\begin{lem}
\label{lem:from_gen_Ganelius}
We have
\begin{equation}
\sup_{x \in (-1,1)} (1-x^2)^{\mu/2} \left| B_N (x ; \boldsymbol{\beta}, d) \right|
\leq 
C_{2} \exp \left (-\sqrt{\pi d \mu N} \right ) \label{eq:general-Ganelius-[-1,1]-2}, 
\end{equation}
where $C_{2}$ is a positive real number independent of $N$.
\end{lem}

Then, Theorem~\ref{thm:general-Ganelius-approximation} is proven as follows. 
We derive from Lemmas~\ref{lem:UB_step1} and~\ref{lem:UB_step2} that  
\begin{align}
& \sup_{
\begin{subarray}{c}
f \in \boldsymbol{H}^{\infty}(\varLambda_{d}, w_{\mu}) \\
\| f \| \leq 1
\end{subarray}
}
\left(
\sup_{x \in (-1,1)}
\left|
f(x) - \tilde{f}_{\nu, N}(x)
\right|
\right) \notag \\
& \leq
\sup_{x \in (-1,1)} \frac{C_{1}}{2\pi} \, (1-x^2)^{\mu/2} \, | B_n (x ; \boldsymbol{\beta}, d) |.
\notag
\end{align}
Then, by Lemma~\ref{lem:from_gen_Ganelius}, we have the estimate 
in~\eqref{eq:general-Ganelius-error} in Theorem~\ref{thm:general-Ganelius-approximation}. 
Finally, by Theorem~\ref{thm:min_err_lower_bound} with $n = 2N$ and Theorem~\ref{thm:general-Ganelius-approximation}, 
we have
\begin{align}
c \, \exp \left( - \sqrt{\pi d \mu N} \right)
& \leq 
E_{2N}^{\mathrm{min}}(\boldsymbol{H}^{\infty}(\varLambda_{d}, w_{\mu})) \notag \\
& \leq 
\sup_{
\begin{subarray}{c}
f \in \boldsymbol{H}^{\infty}(\varLambda_{d}, w_{\mu}) \\
\| f \| \leq 1
\end{subarray}
}
\left(
\sup_{x \in (-1,1)}
\left|
f(x) - \tilde{f}_{\nu, N}(x)
\right|
\right) \notag \\
& \leq
C \exp \left( - \sqrt{\pi d \mu N} \right), 
\notag 
\end{align}
which guarantees the optimality of Formula~\eqref{eq:main_formula}.

\section{Proofs of Lemmas~\ref{lem:UB_step1}--\ref{lem:from_gen_Ganelius}}
\label{sec:proof_of_main_steps}
\subsection{Proof of Lemma~\ref{lem:UB_step1}}

Let $x \in (-1,1)$ and let $\delta$ be a positive number such that $\{ x \} \cup \{ \beta_{k} \} \subset \varLambda_d(\delta)$. 
We begin with writing the difference $f(x) - \tilde{f}_{\nu, N}(x)$ by an complex contour integral. 
By using $g(z) := f(z)/(1-z^{2})^{\nu}$, 
we define $I_{\varLambda_d(\delta)}(x)$ by 
\begin{equation}
I_{\varLambda_d(\delta)}(x) := 
\frac{1}{2\pi \mathrm{i}} \oint_{\partial \varLambda_d(\delta)} \dfrac{1}{B_N (z ; \boldsymbol{\beta}, d)} \, \dfrac{g(z)}{z-x} \, \mathrm{d}z.
\label{eq:tilde-I(x)}
\end{equation}
Because the function $g$ is analytic on $\varLambda_d(\delta)$ and bounded on the closure of $\varLambda_d(\delta)$, 
it follows from the residue theorem that 
\begin{align}
I_{\varLambda_d(\delta)}(x) 
=
\frac{g(x)}{B_{N}(x; \boldsymbol{\beta}, d)} 
- \psum_{k = -N}^{N} g(\beta_{k}) \, \frac{2d \sigma_{k}}{\pi} \, \frac{1 - \beta_{k}^{2}}{x - \beta_{k}}. 
\notag
\end{align}
By multiplying both sides of the above equality by $(1-x^{2})^{\nu} B_{N} (x ; \boldsymbol{\beta}, d)$, we have
\begin{align}
\label{eq:general-pre-g}
& (1-x^2)^{\nu} \, B_{N} (x ; \boldsymbol{\beta}, d) \, I_{\varLambda_d (\delta)}(x) \\
& =
f(x) - \psum_{k=-N}^{N} f(\beta_k) \, \frac{2d \sigma_{k}}{\pi} \,  
\frac{(1-x^2)^{\nu}}{(1-\beta_{k}^{2})^{\nu-1}} \, \frac{B_{N} (x ; \mbox{\boldmath $\beta$}, d)}{x-\beta_k} \notag \\
& = f(x) - \tilde{f}_{\nu, N}(x). \notag 
\end{align}

Then, 
we bound $| I_{\varLambda_d (\delta)}(x) |$ from above
for $f \in \boldsymbol{H}^{\infty}(\varLambda_{d}, w_{\mu})$ with $\| f \| \leq 1$ 
and $x \in (-1,1)$ as follows:
\begin{align}
\label{eq:I_bound}
& \left| I_{\varLambda_d(\delta)}(x) \right | \\
& = 
\frac{1}{2\pi} \left | \oint_{\partial \varLambda_d(\delta)} 
\frac{1}{B_{N} (z ; \boldsymbol{\beta}, d)} \, 
\frac{f(z)}{(1-z^2)^{\mu/2}} \, 
\frac{(1-z^2)^{\mu/2 - \nu}}{z-x} \, \mathrm{d}z \right| \notag \\
&\leq 
\frac{1}{2\pi}
\max_{z \in \partial \varLambda_d(\delta)} \left| \frac{1}{B_{N} (z ; \boldsymbol{\beta}, d)} \right| \, 
\oint_{\partial \varLambda_d(\delta)} 
\left| 
\frac{(1-z^2)^{\mu/2-\nu}}{z-x} 
\right| 
|\mathrm{d}z|. 
\notag
\end{align}

Finally, we can derive the conclusion of Lemma~\ref{lem:UB_step1} 
from~\eqref{eq:general-pre-g}, \eqref{eq:I_bound}, and the fact that
\begin{align}
\notag
\lim_{\delta \to +0} \max_{z \in \partial \varLambda_d(\delta)} \left| \frac{1}{B_{N} (z ; \boldsymbol{\beta}, d)} \right| = 1. 
\end{align}
This equality follows from the expression
\begin{align}
B_{N} (z ; \boldsymbol{\beta}, d) 
& = \pprod_{k=-N}^{N} \tanh \left[ \frac{\pi}{2d} (\zeta/2 - \mathop{\mathrm{arctanh}} \beta_{k} ) \right] \notag \\
& = \pprod_{k=-N}^{N} \tanh 
\left[ \frac{\pi}{2d} (\mathop{\mathrm{Re}} \zeta/2 - \mathop{\mathrm{arctanh}} \beta_{k} ) 
+ \frac{\pi}{4d} (\mathop{\mathrm{Im}} \zeta) \, \mathrm{i} \right], \notag
\end{align}
where $\zeta = 2 \mathop{\mathrm{arctanh}} z \in \mathcal{D}_{d}$ (see \eqref{eq:def_Lambda_d}). 
Note that 
$z \in \partial \varLambda_d \iff | \mathop{\mathrm{Im}} \zeta | = d $ and 
$|\tanh [s \pm (\pi/4)\, \mathrm{i}] | = 1$ for any $s \in \mathbf{R}$. 

\subsection{Proof of Lemma~\ref{lem:UB_step2}}

It suffices to bound 
\begin{align}
\notag 
I(d, \mu, \nu; x) 
:=
\oint_{\partial \varLambda_d} \left | \frac{(1-z^2)^{\mu/2-\nu}}{z-x} \right| \, | \mathrm{d}z |
\end{align}
from above by the RHS of~\eqref{eq:general-Const}. 
Owing to the symmetry of the contour $\partial \varLambda_d$ and the integrand with respect to the real axis, 
we only consider its upper half. 
We employ the variable transformations given by
\begin{align}
\notag 
z = \tanh \left( \frac{s + d\, \ii}{2} \right) \qquad (-\infty < s < \infty) 
\end{align}
and $x = \tanh(t/2)$ with $t \in \mathbf{R}$ to obtain
\begin{align}
\label{eq:tanh_cosh}
& \frac{1}{2} I(d, \mu , \nu ; \tanh(t/2)) \\
& =
\frac{1}{2}
\int_{-\infty}^{\infty}  
\frac{1}{\left| \tanh((s + d\, \ii)/2) - \tanh (t/2) \right|}\,
\frac{1}{\left| \cosh((s + d\, \ii)/2) \right|^{\mu - 2\nu + 2}}\, \dd s. \notag
\end{align}
Because we have 
\begin{align}
\label{eq:tanh_diff_abs}
& \left| \tanh((s + d\, \ii)/2) - \tanh (t/2) \right|^{2}
=
\frac{1}{\cosh^{2}(t/2)}
\frac{ \cosh(s-t) - \cos d }{\cosh s + \cos d}, \\
\label{eq:cosh_abs}
& \left| \cosh((s + d\, \ii)/2) \right|^{2} 
= 
\frac{1}{2} (\cosh s + \cos d)
\end{align}
after some algebra, 
it follows from~\eqref{eq:tanh_cosh}, \eqref{eq:tanh_diff_abs}, and~\eqref{eq:cosh_abs} that 
\begin{align}
\label{eq:tanh_cosh_alg}
& I(d, \mu, \nu ; \tanh(t/2)) \\
& =
\int_{-\infty}^{\infty}  
\frac{2^{(\mu - 2\nu + 2)/2} \cosh(t/2)}{(\cosh(s-t) - \cos d)^{1/2}\, (\cosh s + \cos d)^{(\mu-2\nu + 1)/2}}
\, \dd s. \notag
\end{align}

Then, by noting that $|\cos d \,| < 1$ and that 
\begin{align}
\notag 
(1 - |\cos d \, |) \cosh u
\leq \cosh u \pm \cos d 
\leq (1 + |\cos d \, |) \cosh u
\end{align}
holds for a real number $u$, 
we can derive from~\eqref{eq:tanh_cosh_alg} that 
\begin{align}
I(d, \mu, \nu; \tanh(t/2))
\leq
U_{d, \mu, \nu} 
\, \cosh(t/2)
\, J(1/2, (\mu-2\nu + 1)/2; t), 
\label{eq:tanh_cosh_alg_ineq}
\end{align}
where 
\begin{align}
\notag 
& U_{d, \mu, \nu} 
:= 2^{(\mu - 2\nu + 2)/2} 
\max \left\{ 
\frac{1}{ (1 - |\cos d \, |)^{(\mu-2\nu + 2)/2} } , 
\frac{(1 + |\cos d \, |)^{(2\nu - 1 -\mu)/2}}{ (1 - |\cos d \, |)^{1/2} } 
\right\},  
\end{align}
and
\begin{align}
\label{eq:def_J}
& J(\alpha, \beta ; t) 
:=
\int_{-\infty}^{\infty}
\frac{1}{\cosh^{\alpha}(s-t)\, \cosh^{\beta} s }\, \dd s. 
\end{align}
Therefore, what remains is to estimate $J(\alpha, \beta ; t)$ given by~\eqref{eq:def_J}. 
This estimate is done as shown in the following lemma also used in~\cite{bib:UTOS_Ganelius_JSIAM}. 
For readers' convenience, we describe its proof in Appendix~\ref{sec:proof_of_lem_coshcosh}.

\begin{lem}[{\cite[Lemma 3.4]{bib:UTOS_Ganelius_JSIAM}}]
\label{lem:coshcosh}
Let $\alpha$ and $\beta$ be distinct real numbers with $\alpha+\beta > 0$ and 
let $a = \max \{ \alpha, \beta \}$ and $b = \min \{ \alpha , \beta \}$. 
Then, for a real number $t$, we have
\begin{align}
J(\alpha, \beta; t) \leq
\frac{\max\{ 2^{a+1}, 2^{a+b+1} \}}{a^{2} - b^{2}} \left( -b\, \ee^{-a|t|} + a\, \ee^{-b|t|} \right). 
\end{align}
\end{lem}

By Lemma~\ref{lem:coshcosh}, we have
\begin{align}
J(1/2, (\mu-2\nu + 1)/2; t)
\leq 
\tilde{C}_{\mu, \nu} \, \{ (2\nu - 1 - \mu)\, \ee^{-|t|/2} + \ee^{-((\mu-2\nu + 1)/2) |t|} \}, 
\label{eq:J_estim}
\end{align}
where $\tilde{C}_{\mu, \nu}$ is a positive number depending only on $\mu$ and $\nu$. 
Therefore, it follows from \eqref{eq:tanh_cosh_alg_ineq} and \eqref{eq:J_estim} that 
\begin{align}
I(d,\mu, \nu; \tanh(t/2))
\leq
\hat{C}_{d, \mu, \nu} \, \left\{ \cosh^{2}(t/2) \right\}^{-(\mu/2-\nu)}, 
\label{eq:pre_final_2}
\end{align}
where $\hat{C}_{d, \mu, \nu}$ is a positive number depending only on $d$, $\mu$, and $\nu$. 
By converting Inequality~\eqref{eq:pre_final_2} into that with respect to $x$ by the transform
$t = 2 \mathop{\mathrm{arctanh}} x = \log((1+x)/(1-x))$, 
we obtain the concluding inequality in Lemma~\ref{lem:UB_step2}. 

\subsection{Proof of Lemma~\ref{lem:from_gen_Ganelius}}

We reduce Lemma~\ref{lem:from_gen_Ganelius} to the following theorem, 
which is a generalization of Theorem~\ref{thm:OldGan} (\cite[Lemma 1]{bib:JangHaber2001}). 
We prove this theorem in Appendix~\ref{sec:proof_of_generalGan}. 

\begin{thm}
\label{thm:GenGan}
Let $r$ be a positive real number and let $N$ be a positive integer. 
Furthermore, let $\{ a_{k} \}$ be the sequence of the modified Ganelius sampling points given by Definition~\ref{defn:Ganelius_points}. Then, 
\begin{align}
\max_{s \in [0,1]} s^{r} \prod_{k = 1}^{N} \left| \frac{s-a_{k}}{s+a_{k}} \right| \leq C \exp \left( - \pi \sqrt{N r} \right)
\label{eq:mod_Gan}
\end{align}
holds, where $C$ is a positive number independent of $N$. 
\end{thm}

\begin{rem}
We just assume that $r > 0$ in Theorem~\ref{thm:GenGan} as opposed to Theorem~\ref{thm:OldGan}. 
\end{rem}

As stated right after Lemma~1 in~\cite{bib:JangHaber2001}, 
owing to the variable transformation 
\begin{align}
s=\dfrac{1-t^2}{1+t^2}, 
\end{align}
Inequality~\eqref{eq:mod_Gan} in Theorem~\ref{thm:GenGan} is equivalent to
\begin{equation}
\max_{t \in [-1, 1]} 
(1-t^{2})^{r} \pprod_{k=-N}^{N} \left| \frac{t - b_{k}}{1 - b_{k} t} \right| 
\leq 
C_{3}' \exp \left( -\pi \sqrt{N r} \right), 
\label{eq:GaneliusTheoryInfinity2}
\end{equation}
where $C_{3}'$ is a positive number independent of $N$. 
In the following, 
we reduce Inequality~\eqref{eq:general-Ganelius-[-1,1]-2} in Lemma~\ref{lem:from_gen_Ganelius} 
to Inequality~\eqref{eq:GaneliusTheoryInfinity2}. 

First, by letting 
\begin{align}
x = \tanh \left( \frac{2d}{\pi} \mathop{\mathrm{arctanh}} t \right)
\label{eq:xz_trans}
\end{align}
for $t \in [-1, 1]$, we have
\begin{align}
B_{N}(x; \boldsymbol{\beta}, d) = \pprod_{k=-N}^{N} \frac{t - b_{k}}{1 - b_{k} t}.  
\notag
\end{align}
Next, by noting that 
$\mathop{\mathrm{arctanh}} t = (1/2)\log((1+t)/(1-t))$, we have
\begin{align}
1-x^{2} 
& = 
\frac{1}{\cosh^{2}((2d/\pi) \mathop{\mathrm{arctanh}} t )} \notag \\
& =
\left(
\frac{2\, (1-t^{2})^{d/\pi}}{(1+t)^{2d/\pi} + (1-t)^{2d/\pi}}
\right)^{2} \notag \\
& \leq 
2^{2(2d/\pi + 1)}\, (1-t^{2})^{2d/\pi}, 
\notag 
\end{align}
where we employ the inequality
\(
A^{\alpha} + B^{\alpha} \geq \left[ (A + B)/2 \right]^{\alpha}
\)
that holds for any positive numbers $\alpha$, $A$, and $B$. 
From these and~\eqref{eq:GaneliusTheoryInfinity2}, 
by letting $r = d \mu/\pi$, 
we have
\begin{align}
\notag
(1 - x^{2})^{\mu/2} | B_{N}(x ; \boldsymbol{\beta}, d) | 
& \leq 
2^{(2d/\pi + 1)\mu}\, (1-t^{2})^{d \mu/\pi} 
\pprod_{k=-N}^{N} \left| \frac{t - b_{k}}{1 - b_{k} t} \right| \notag \\
& \leq C_{3}'' \exp \left( - \sqrt{\pi d \mu N} \right), \notag
\end{align}
where $C_{3}''$ is a positive number independent of $N$. 
Thus Lemma~\ref{lem:from_gen_Ganelius} is proven.

\section{Numerical experiments}
\label{sec:num_exp}
In this section, 
we compute the approximations of some functions by Formula~\eqref{eq:main_formula} and observe their errors.
Moreover, 
we compare the errors with those of the SE-Sinc approximations given by Formula~\eqref{eq:SE-Sinc} for the same functions. 

We adopt the following functions for this numerical experiment. 
\begin{align*}
& \text{Example 1} \qquad 
f_{1}(x) := \sqrt{\frac{1-x^{2}}{1+x^{2}}}, \\
& \text{Example 2} \qquad 
f_{2}(x) := \sqrt{\frac{3-3x^{2}}{1+3x^{2}}}, \\
& \text{Example 3} \qquad 
f_{3}(x) := \sqrt{\frac{1-x^{2}}{3+x^{2}}}, \\
& \text{Example 4} \qquad 
f_{4}(x) := (1-x^{2})^{1/\sqrt{2}} \sqrt{\cos(4 \mathop{\mathrm{arctanh}} x) + \cosh \pi}, \\
& \text{Example 5} \qquad 
f_{5}(x) := \left( \frac{1-x^{2}}{1+x^{2}} \right)^{3/2}. 
\end{align*}
These functions have the singularities at $x = \pm 1$. 
Table~\ref{tab:S_d_mu} shows the other singularities of $f_{1}, \ldots, f_{5}$ and 
the parameters $d$ and $\mu$ such that $f_{i} \in \boldsymbol{H}^{\infty}(\mathcal{D}_{d}, w_{\mu})$ for $i = 1,\ldots, 5$.
We can adopt arbitrary positive values of $\varepsilon_{1}, \varepsilon_{2}, \varepsilon_{3}$, and $\varepsilon_{5}$ 
as long as $d > 0$ in Table~\ref{tab:S_d_mu}. 
In this experiment, we set
$\varepsilon_{1} = \varepsilon_{5} = \pi/2 - 1.57$, 
$\varepsilon_{2} = \pi/3 - 1.047$, and
$\varepsilon_{3} = 2\pi/3 - 2.094$. 
The function $f_{5}$ is an example that does not satisfy the old condition $\mu < \min \{ 2, \pi /d \}$ 
assumed in Theorem~\ref{thm:UTOS_Ganelius_JSIAM}. 

\begin{table}[!ht]
\caption{The singularities other than $\pm 1$ and parameters $d$ and $\mu$ of $f_{1}, \ldots, f_{5}$. 
The positive numbers $\varepsilon_{1}, \varepsilon_{2}, \varepsilon_{3}$, and $\varepsilon_{5}$ are arbitrary as long as $d > 0$. }
\label{tab:S_d_mu}
\begin{tabular}{l|l|l|l}
& Singularities & $d$ & $\mu$ \\
\hline
$f_{1}$ & $\pm \mathrm{i}$ & $\pi/2 - \varepsilon_{1}$ & $1$ \\
$f_{2}$ & $\pm \mathrm{i}/\sqrt{3}$ & $\pi/3 - \varepsilon_{2}$ & $1$ \\
$f_{3}$ & $\pm \mathrm{i} \sqrt{3}$ & $2\pi/3 - \varepsilon_{3}$ & $1$ \\
$f_{4}$ & $\tanh[m+(1\pm \mathrm{i})\pi/2)/2] \quad (m \in \mathbf{Z})$ & $\pi/2$ & $\sqrt{2}$ \\
$f_{5}$ & $\pm \mathrm{i}$ & $\pi/2 - \varepsilon_{5}$ & $3$
\end{tabular}
\end{table} 

In order to use Formula~\eqref{eq:main_formula}, 
we need to decide the value of $\nu$ in $\tilde{f}_{\nu, N}$. 
We set $\nu = \lceil \mu/2 \rceil$, i.e., $\nu = 1$ for $f_{1}, \ldots, f_{4}$ and $\nu = 2$ for $f_{5}$.
Then, for a fixed $N$ and each of the functions $f = f_{1}, \ldots, f_{5}$, 
we compute the values of $f(x) - \tilde{f}_{\nu, N}(x)$ for $x \in X \cup Y \subset (-1,1)$, 
where
\begin{align}
& X := \{ i/1000 \mid i = -999, \ldots, 999 \}, \notag \\
& Y := \{ \pm (1 - k/10^{\ell}) \mid \ell = 4,\ldots, 16, \ k = 1,\ldots, 9 \}, \notag 
\end{align}
and adopt the maximum of their absolute values as the computed error. 
Because the computed error is often attained at the points close to the endpoints $\pm 1$, 
we employ the set $Y$ in order to capture those points. 
For all the computations, 
we used a computer with PowerPC G5 Dual 2 GHz CPU and GCC 4.0.1 compiler, 
and programs written in C with all the floating point numbers declared as the ``long double'' variables. 
Then, all the computations are done with the quadruple precision floating point numbers. 

We show the computed errors for $N = 4, 9, 16, \ldots, 144$ 
in Figures~\ref{fig:example1}--\ref{fig:example5}.
In each figure, the legends ``SE-Sinc'' and ``Ganelius'' indicate the results of 
Formulas~\eqref{eq:SE-Sinc} and~\eqref{eq:main_formula}, respectively. 
Note that the total number $n$ of the sampling points is 
$n = 2N+1$ for Formula~\eqref{eq:SE-Sinc} and
$n = 2N$ for Formula~\eqref{eq:main_formula}. 
Furthermore, we estimate the decay rate of the errors by computing the ratio
$(\text{the error for $N = (m-1)^{2}$}) / (\text{the error for $N = m^{2}$})$ for $m = 2,3,4, \ldots, 12$. 
The theoretical values of the ratio for Formulas~\eqref{eq:SE-Sinc} and~\eqref{eq:main_formula}
are $\exp(\sqrt{\pi d \mu / 2})$ and $\exp(\sqrt{\pi d \mu})$, respectively. 
We show the computed ratios (``rate'') and theoretical values (``t.v.'') in Tables~\ref{tab:example1}--\ref{tab:example5}.

\begin{figure}[!ht]
  \begin{minipage}[c]{0.49\textwidth}
  	 \hspace{1cm}
  	 \begin{center} 
     \includegraphics[width=\linewidth]{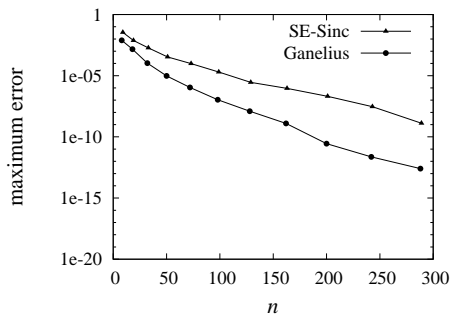}
     \end{center} 
     \vspace{0.5cm}
     \figcaption{Maximum error in the approximation of $f_{1}$. }
     \label{fig:example1}
  \end{minipage}
  \begin{minipage}[c]{0.49\textwidth}
    \begin{center}
    \tblcaption{Convergence rate of the approximation of $f_{1}$. \protect\\}
    \label{tab:example1}
    \footnotesize
      \begin{tabular}{rlrlr} \hline
   $N$ & Ganelius & rate & SE-Sinc & rate \\ \hline
       4 & 7.73 E$^{-3}$ &   & 3.48 E$^{-2}$ &    \\
       9 & 1.47 E$^{-3}$ & 5.2 & 7.49 E$^{-3}$ &  4.6 \\
     16 & 1.06 E$^{-4}$ & 13.8 & 1.88 E$^{-3}$ &  3.9 \\
     25 & 9.57 E$^{-6}$ & 11.0 & 3.38 E$^{-4}$ &  5.5 \\
     36 & 1.10 E$^{-6}$ & 8.6 & 9.67 E$^{-5}$ &  3.5 \\
     49 & 1.07 E$^{-7}$ & 10.2 & 1.98 E$^{-5}$ &  4.8 \\
     64 & 1.25 E$^{-8}$ & 8.5 & 2.85 E$^{-6}$ &  6.9 \\
     81 & 1.25 E$^{-9}$ & 9.9 & 9.23 E$^{-7}$ &  3.0 \\
   100 & 2.78 E$^{-11}$ & 45.1 & 2.04 E$^{-7}$ &  4.5 \\
   121 & 2.31 E$^{-12}$ & 12.0 & 2.92 E$^{-8}$ &  6.9 \\
   144 & 2.55 E$^{-13}$ & 9.0 & 1.30 E$^{-9}$ &  22.3 \\
   t.v. &           &  9.2 &            &  4.8 \\ \hline
      \end{tabular}
    \end{center}
  \end{minipage}
\end{figure}

\begin{figure}[!ht]
  \begin{minipage}[c]{0.49\textwidth}
  	\hspace{1cm}
  	 \begin{center} 
    \includegraphics[width=\linewidth]{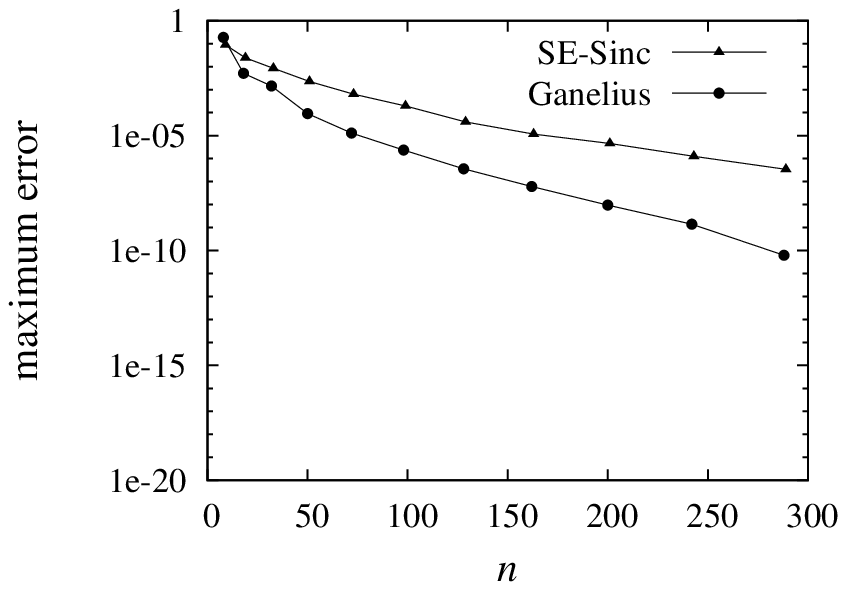}
    \end{center} 
    \vspace{0.5cm}
    \figcaption{Maximum error in the approximation of $f_{2}$.}
    \label{fig:example2}
  \end{minipage}
  \begin{minipage}[c]{.49\textwidth}
    \begin{center}
    \tblcaption{Convergence rate of in the approximation of $f_{2}$. \protect\\}
    \label{tab:example2}
    \footnotesize
      \begin{tabular}{rlrlr} \hline
   $N$ & Ganelius & rate & SE-Sinc & rate \\ \hline
       4 & 1.89 E$^{-1}$ &   & 8.96 E$^{-2}$ &    \\
       9 & 5.17 E$^{-3}$ & 36.5 & 2.40 E$^{-2}$ &  3.7 \\
     16 & 1.44 E$^{-3}$ & 3.5 & 8.56 E$^{-3}$ &  2.8 \\
     25 & 9.13 E$^{-5}$ & 15.8 & 2.27 E$^{-3}$ & 3.7 \\
     36 & 1.28 E$^{-5}$ & 7.1 & 6.41 E$^{-4}$ &  3.5 \\
     49 & 2.34 E$^{-6}$ & 5.4 & 1.94 E$^{-4}$ &  3.3 \\
     64 & 3.57 E$^{-7}$ & 6.5 & 3.91 E$^{-5}$ &  4.9 \\
     81 & 6.06 E$^{-8}$ & 5.8 & 1.15 E$^{-5}$ &  3.3 \\
   100 & 9.46 E$^{-9}$ & 6.4 & 4.58 E$^{-6}$ &  2.5 \\
   121 & 1.40 E$^{-9}$ & 6.7 & 1.25 E$^{-6}$ &  3.6 \\
   144 & 6.17 E$^{-11}$ & 22.7 & 3.39 E$^{-7}$ & 3.6 \\
   t.v. &           &  6.1  &               &  3.6 \\ \hline
      \end{tabular}
    \end{center}
  \end{minipage}
\end{figure}

\begin{figure}[!ht]
  \begin{minipage}[c]{0.49\textwidth}
    \hspace{1cm}
    \begin{center} 
    \includegraphics[width=\linewidth]{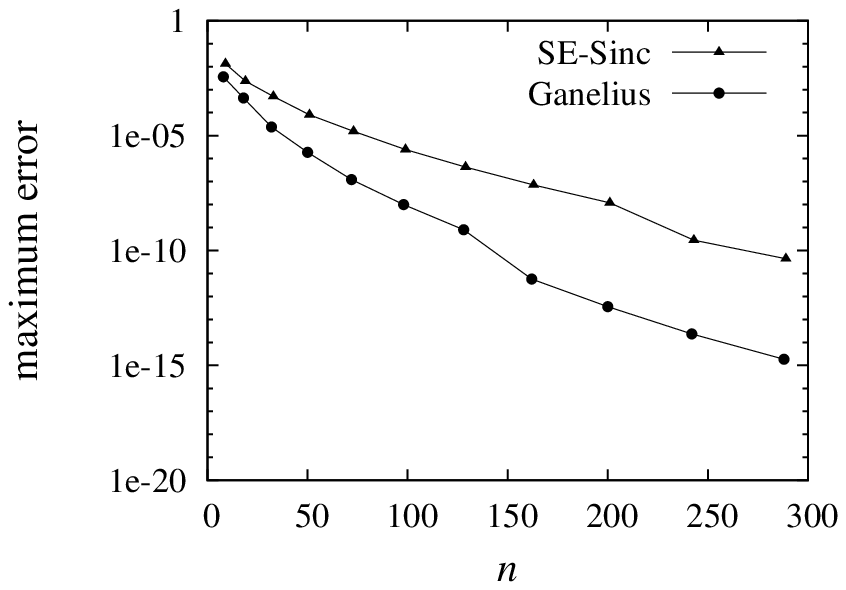}
    \end{center} 
    \vspace{0.5cm}
    \figcaption{Maximum error in  in the approximation of $f_{3}$. }
    \label{fig:example3}
  \end{minipage}
  \begin{minipage}[c]{.49\textwidth}
    \begin{center}
    \tblcaption{Convergence rate of in the approximation of $f_{3}$. \protect\\}
    \label{tab:example3}
    \footnotesize
      \begin{tabular}{rlrlr} \hline
   $N$ & Ganelius & rate & SE-Sinc & rate \\ \hline
       4 & 3.63 E$^{-3}$ &   & 1.33 E$^{-2}$ &    \\
       9 & 4.35 E$^{-4}$ & 8.3 & 2.33 E$^{-3}$ &  5.7 \\
     16 & 2.36 E$^{-5}$ & 18.3 & 5.06 E$^{-4}$ &  4.6 \\
     25 & 1.85 E$^{-6}$ & 12.7 & 8.04 E$^{-5}$ &  6.2 \\
     36 & 1.22 E$^{-7}$ & 15.1 & 1.52 E$^{-5}$ &  5.2 \\
     49 & 1.00 E$^{-8}$ & 12.2 & 2.49 E$^{-6}$ &  6.1 \\
     64 & 7.97 E$^{-10}$ & 12.5 & 4.25 E$^{-7}$ &  5.8 \\
     81 & 5.76 E$^{-12}$ & 138.4 & 7.14 E$^{-8}$ &  5.9 \\
   100 & 3.60 E$^{-13}$ & 15.9 & 1.17 E$^{-8}$ &  6.0 \\
   121 & 2.33 E$^{-14}$ & 15.4 & 2.82 E$^{-10}$ &  41.8 \\
   144 & 1.83 E$^{-15}$ & 12.7 & 4.39 E$^{-11}$ &  6.4 \\
   t.v. &           &  13.0  &         &  6.1 \\ \hline
      \end{tabular}
    \end{center}
  \end{minipage}
\end{figure}

\begin{figure}[!ht]
  \begin{minipage}[c]{0.49\textwidth}
    \hspace{1cm}    
  	\begin{center} 
    \includegraphics[width=\linewidth]{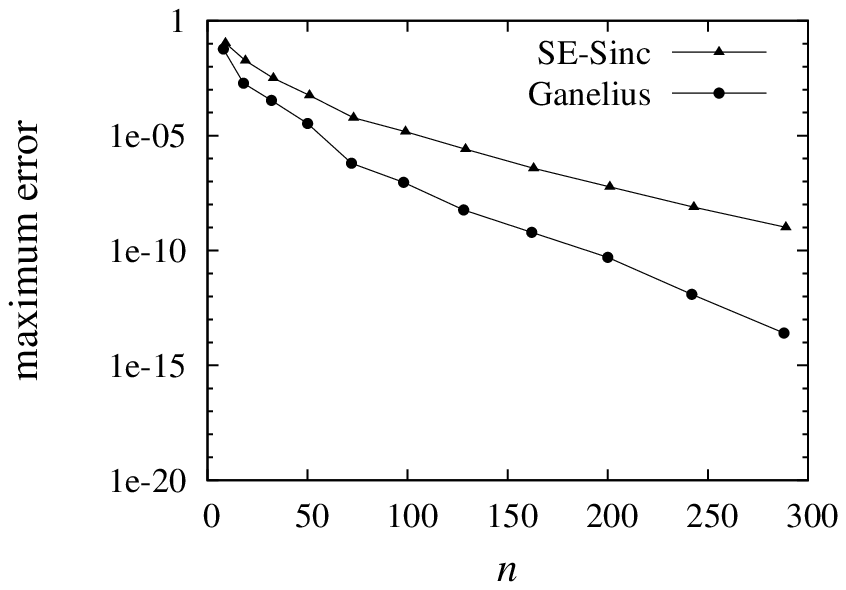}
    \end{center} 
    \vspace{0.5cm}
    \figcaption{Maximum error in the approximation of $f_{4}$. }
    \label{fig:example4}
  \end{minipage}
  \begin{minipage}[c]{.49\textwidth}
    \begin{center}
    \tblcaption{Convergence rate of the approximation of $f_{4}$. \protect\\}
    \label{tab:example4}
    \footnotesize
      \begin{tabular}{rlrlr} \hline
   $N$ & Ganelius & rate & SE-Sinc & rate \\ \hline
       4 & 5.83 E$^{-2}$ &   & 1.06 E$^{-1}$ &    \\
       9 & 1.90 E$^{-3}$ & 30.6 & 1.81 E$^{-2}$ &  5.8 \\
     16 & 3.41 E$^{-4}$ & 5.5 & 3.14 E$^{-3}$ &  5.7 \\
     25 & 3.35 E$^{-5}$ & 10.1 & 5.59 E$^{-4}$ &  5.6 \\
     36 & 6.26 E$^{-7}$ & 53.6 & 5.95 E$^{-5}$ &  9.3 \\
     49 & 9.30 E$^{-8}$ & 6.7 & 1.47 E$^{-5}$ &  4.0 \\
     64 & 5.77 E$^{-9}$ & 16.0 & 2.54 E$^{-6}$ &  5.7 \\
     81 & 6.14 E$^{-10}$ & 9.3 & 3.78 E$^{-7}$ &  6.7 \\
   100 & 5.04 E$^{-11}$ & 12.1 & 5.88 E$^{-8}$ &  6.4 \\
   121 & 1.23 E$^{-12}$ & 40.9 & 7.63 E$^{-9}$ &  7.7 \\
   144 & 2.55 E$^{-14}$ & 48.1 & 1.01 E$^{-9}$ &  7.5 \\
   t.v. &           &  14.0  &         &  6.5 \\ \hline
      \end{tabular}
    \end{center}
  \end{minipage}
\end{figure}

\begin{figure}[!ht]
  \begin{minipage}[c]{0.49\textwidth}
    \hspace{1cm}    
  	\begin{center} 
    \includegraphics[width=\linewidth]{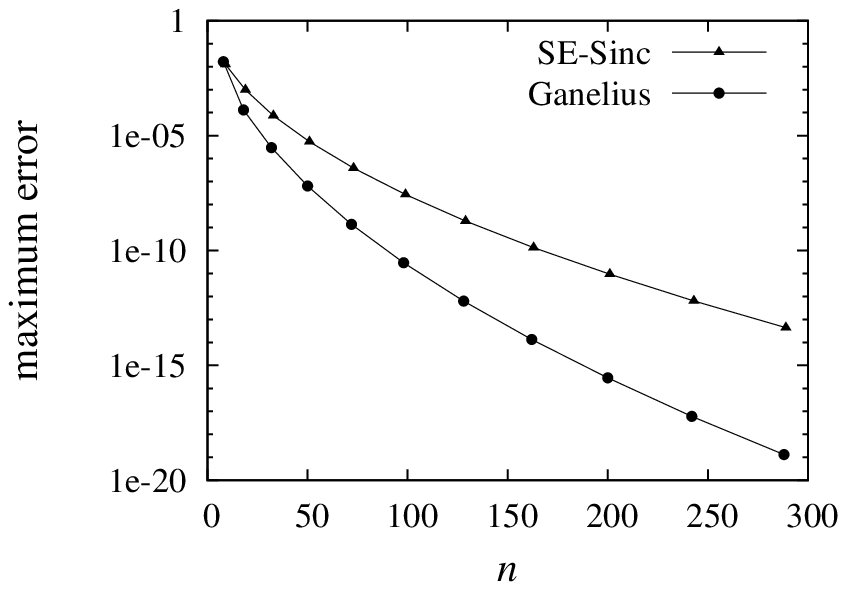}
    \end{center} 
    \vspace{0.5cm}
    \figcaption{Maximum error in the approximation of $f_{5}$. }
    \label{fig:example5}
  \end{minipage}
  \begin{minipage}[c]{.49\textwidth}
    \begin{center}
    \tblcaption{Convergence rate of the approximation of $f_{5}$. \protect\\}
    \label{tab:example5}
    \footnotesize
      \begin{tabular}{rlrlr} \hline
   $N$ & Ganelius & rate & SE-Sinc & rate \\ \hline
       4 &  1.64 E$^{-2}$ &   &  1.24 E$^{-2}$ &    \\
       9 &  1.30 E$^{-4}$ & 126.4 &  9.91 E$^{-4}$ &  12.5 \\
     16 &  2.98 E$^{-6}$ & 43.6  &  7.37 E$^{-5}$ & 13.4  \\
     25 &  6.43 E$^{-8}$ & 46.4 &  5.38 E$^{-6}$ & 13.6  \\
     36 &  1.38 E$^{-9}$ & 46.6 &  3.85 E$^{-7}$ & 13.9  \\
     49 &  2.93 E$^{-11}$ & 46.9 &  2.72 E$^{-8}$ &  14.1 \\
     64 &  6.29 E$^{-13}$ & 46.6 &  1.91 E$^{-9}$ &  14.2 \\
     81 &  1.33 E$^{-14}$ & 47.0 &  1.33 E$^{-10}$ & 14.3  \\
   100 &  2.85 E$^{-16}$ & 46.9 &  9.23 E$^{-12}$ & 14.4  \\
   121 &  6.06 E$^{-18}$ & 46.9 &  6.36 E$^{-13}$ & 14.5  \\
   144 &  1.30 E$^{-19}$ & 46.4 &  4.36 E$^{-14}$ & 14.5  \\
   t.v. &           &  46.8  &         & 15.1 \\ \hline
      \end{tabular}
    \end{center}
  \end{minipage}
\end{figure}

From these results, we can observe that in each case 
Formula~\eqref{eq:main_formula} outperforms Formula~\eqref{eq:SE-Sinc}
and the sequence of the computed values of the ``rate'' approaches its theoretical value as $N$ increases.

\section{Concluding remarks}
\label{sec:concl}
In this paper, we proposed the approximation formula given by~\eqref{eq:main_formula} that is optimal 
in the space $\boldsymbol{H}^{\infty}(\mathcal{D}_{d}, w_{\mu})$ for any $\mu > 0$. 
Formula~\eqref{eq:main_formula} is a generalization of Formula~\eqref{eq:old_formula} 
proposed by Ushima et al.~\cite{bib:UTOS_Ganelius_JSIAM}, 
which is valid only in the case that $\mu < \min\{ 2, \pi/d \}$. 
In order to estimate the error of Formula~\eqref{eq:main_formula}, 
we showed Theorem~\ref{thm:GenGan}, 
a generalization of Theorem~\ref{thm:OldGan} (the Ganelius theorem \cite[Lemma 1]{bib:JangHaber2001}). 
Then, we gave an upper bound of the error of Formula~\eqref{eq:main_formula} in 
Theorem~\ref{thm:general-Ganelius-approximation} and 
showed the optimality of the formula by combining this upper bound
and the lower bound of the minimum error norm given by 
Theorem~\ref{thm:min_err_lower_bound} ({\cite[(a) on page 782]{bib:Sugihara2003_OptSinc}}). 
Furthermore, we observed that Formula~\eqref{eq:main_formula} achieved 
the optimal convergence rate in the numerical experiment. 

We can list some themes for future work about the optimal approximation in 
$\boldsymbol{H}^{\infty}(\varLambda_{d}, w_{\mu})$: 
finding other optimal formulas and comparing them with Formula~\eqref{eq:main_formula}, 
inventing improved methods for fast computation by Formula~\eqref{eq:main_formula}, 
applying Formula~\eqref{eq:main_formula} to differential equations such as two point boundary problems, etc.

\appendix

\section{Proof of Lemma~\ref{lem:coshcosh}}
\label{sec:proof_of_lem_coshcosh}
Without loss of generality, we can assume that $t \geq 0$, $a = \alpha$, and $b = \beta$. 
The assumptions about $a$ and $b$ are owing to the fact that 
\begin{align}
\notag 
J(\alpha, \beta ; t) 
& =
\int_{-\infty}^{\infty}
\frac{1}{\cosh^{\alpha}(t-s)\, \cosh^{\beta} s }\, \dd s \notag \\
& = 
\int_{-\infty}^{\infty}
\frac{1}{\cosh^{\alpha}u \, \cosh^{\beta} (t - u) }\, \dd u
=
J(\beta, \alpha; t). \notag
\end{align}
For the proof of this lemma, we employ the following inequalities:
\begin{align}
\begin{cases}
\ee^{u}/2 \leq \cosh u \leq \ee^{u} & (u \geq 0), \\
\ee^{-u}/2 \leq \cosh u \leq \ee^{-u} & (u < 0). 
\end{cases}
\label{eq:cosh_cases}
\end{align}
In the following, we deal with two cases: (i) $b < 0$ and (ii) $b \geq 0$. 

\bigskip

\noindent
\underline{Case (i)} \\

It holds that $a > 0 > b$ in this case. 
By using Inequalities~\eqref{eq:cosh_cases}, we have
\begin{align}
\notag
\frac{(\cosh s)^{-b}}{\cosh^{a}(s-t)} 
\leq 
\begin{cases}
2^{a} \, \ee^{-a t} \, \ee^{(a+b) s} & (s < 0), \\
2^{a} \, \ee^{-a t} \, \ee^{(a-b) s} & (0 \leq s \leq t), \\
2^{a} \, \ee^{ a t} \, \ee^{-(a+b)s} & (t < s). 
\end{cases}
\end{align}
Therefore, $J(a,b;t)$ is bounded from above as follows: 
\begin{align}
\notag
J(a,b;t) & = 
\left( \int_{-\infty}^{0} + \int_{0}^{t} + \int_{t}^{\infty}   \right)
\frac{(\cosh s)^{-b}}{\cosh^{a}(s-t)}
\, \dd s \\
& \leq
2^{b} \left[
\ee^{-at} \left\{ \frac{1}{a+b} + \frac{1}{a-b} \left( \ee^{(a-b) t} - 1 \right) \right\} + \frac{1}{a+b} \ee^{-bt} 
\right] \notag \\
& = 
\frac{2^{a+1}}{a^{2} - b^{2}}
\left(
-b\, \ee^{-a t} + a \ee^{-b t}
\right). \notag
\end{align}

\bigskip

\noindent
\underline{Case (ii)} \\

It holds that $a \geq b \geq 0$ in this case. 
By using Inequalities~\eqref{eq:cosh_cases}, we have
\begin{align}
\frac{1}{\cosh^{a}(s-t) \, \cosh^{b} s} 
\leq
\begin{cases}
2^{a+b} \, \ee^{-a t} \, \ee^{(a+b) s} & (s < 0), \\
2^{a+b} \, \ee^{-a t} \, \ee^{(a-b) s} & (0 \leq s \leq t), \\
2^{a+b} \, \ee^{ a t} \, \ee^{-(a+b)s} & (t < s). 
\end{cases}
\end{align}
Therefore, 
in the same manner as Case (i), 
$J(a,b;t)$ is bounded from above as follows: 
\begin{align}
J(a,b;t) 
& = 
\left( \int_{-\infty}^{0} + \int_{0}^{t} + \int_{t}^{\infty}   \right)
\frac{1}{\cosh^{a}(s-t) \, \cosh^{b} s}
\, \dd s \notag \\
& \leq
\frac{2^{a+b+1}}{a^{2} - b^{2}}
\left(
-b\, \ee^{-a t} + a \ee^{-b t}
\right). \notag
\end{align}

From the estimates in Cases (i) and (ii), Lemma~\ref{lem:coshcosh} is proven.

\section{Proof of Theorem~\ref{thm:GenGan}}
\label{sec:proof_of_generalGan}
It is sufficient to consider the case that $N$ is larger than a certain positive constant. 
Therefore, we assume that $N$ is large enough such that $N_{0} > 3$.
The conclusion of Theorem~\ref{thm:GenGan} is equivalent to the statement that
\begin{align}
\sum_{k = 1}^{N} \log \left| \frac{s+a_{k}}{s-a_{k}} \right| 
\geq  
\pi \sqrt{N r} + r \log s + C'
\label{eq:mod_Gan_equiv}
\end{align}
holds for any $s \in [0,1]$, where $C'$ is a real number independent of $N$ and $s$. 
We show this statement by proving the following three lemmas. 

\begin{lem}
\label{lem:disc_plus}
On the same assumption as Theorem~\ref{thm:GenGan}, 
for all $s \in [0, 1]$, we have
\begin{align}
\sum_{k = N_{0} + 2}^{N} \log \left| \frac{s + a_{k}}{s - a_{k}} \right| 
\geq 
\frac{\pi \sqrt{N_{0} r}}{2}\, s.
\label{eq:disc_plus}
\end{align}
\end{lem}
\begin{lem}
\label{lem:cont_integral}
On the same assumption as Theorem~\ref{thm:GenGan}, 
for all $s \in [0, 1]$, we have
\begin{align}
\int_{1/\varphi(N_{0})}^{1} \log \left| \frac{s + t}{s - t} \right| \frac{2r \log(t\, \varphi(N_{0}))}{\pi^{2} t} \, \dd t
\geq 
\pi \sqrt{N_{0} r} + r \log s - \frac{\pi \sqrt{N_{0} r}}{2}\, s + C'',
\label{eq:cont_integral}
\end{align}
where $C''$ is a real number independent of $N_{0}$ and $s$. 
\end{lem}
\begin{lem}
\label{lem:cont_vs_disc}
On the same assumption as Theorem~\ref{thm:GenGan}, 
for all $s \in [0, 1]$, we have
\begin{align}
\sum_{k = 1}^{N_{0}+1} \log \left| \frac{s + a_{k}}{s - a_{k}} \right| 
-
\int_{1/\varphi(N_{0})}^{1} \log \left| \frac{s + t}{s - t} \right| \frac{2r \log(t\, \varphi(N_{0}))}{\pi^{2} t} \, \dd t
\geq C''',
\label{eq:cont_vs_disc}
\end{align}
where $C'''$ is a real number independent of $N_{0}$ and $s$. 
\end{lem}

If Lemmas~\ref{lem:disc_plus}--\ref{lem:cont_vs_disc} are proven, 
by adding Inequalities \eqref{eq:disc_plus}--\eqref{eq:cont_vs_disc}, 
we have
\begin{align}
\sum_{k = 1}^{N} \log \left| \frac{s+a_{k}}{s-a_{k}} \right| 
\geq
\pi \sqrt{N_{0} r} + r \log s + (C'' + C''').
\notag 
\end{align}
Furthermore, we have the inequality
\begin{align}
\pi \sqrt{N_{0} r} \geq \pi \sqrt{N r} - c, 
\notag 
\end{align}
where $c$ is a real number independent of $N$. 
By using these inequalities, we can obtain Inequality~\eqref{eq:mod_Gan_equiv}. 

Lemmas~\ref{lem:disc_plus} and~\ref{lem:cont_integral} 
have already been proven by Jang \& Haber \cite{bib:JangHaber2001} 
in the proof on page 221 of \cite{bib:JangHaber2001} 
and ``Proof of (2)'' on pages~218--219 of \cite{bib:JangHaber2001}, respectively, 
without using the assumption $r < 1$. 
On the other hand, 
Inequality \eqref{eq:cont_vs_disc} in Lemma~\ref{lem:cont_vs_disc} corresponds to Inequality (3) on page 218 of \cite{bib:JangHaber2001}, 
where the assumption $r < 1$ is employed. 
Therefore,  we need to prove Lemma~\ref{lem:cont_vs_disc} by ourselves 
in order to remove this assumption. 

\subsection{Proof of Lemma~\ref{lem:cont_vs_disc}}

For $s = 0$, the LHS of \eqref{eq:cont_vs_disc} is zero. 
Furthermore, for $s = a_{k}$, the first term in the LHS is infinity. 
Therefore,  we assume that $s \in (0,1] \setminus \{ a_{k} \}$ in the rest of this section. 
By letting
\begin{align}
g(s, t) := \log \left| \frac{s+t}{s-t} \right|, 
\notag
\end{align}
we have
\begin{align}
\log \left| \frac{s+a_{k}}{s-a_{k}} \right| 
= 
\begin{cases}
g(s\, \varphi(N_{0}), \, \varphi(k-1)) & (k = 1, 2, \ldots, N_{0}), \\
g(s\, \varphi(N_{0}), \, \varphi(k-3/2)) & (k = N_{0} + 1),
\end{cases}
\notag
\end{align}
and 
\begin{align}
\int_{1/\varphi(N_{0})}^{1} \log \left| \frac{s + t}{s - t} \right| \frac{2r \log(t\, \varphi(N_{0}))}{\pi^{2} t} \, \dd t
= 
\int_{0}^{N_{0}} g(s\, \varphi(N_{0}), \, \varphi(u)) \, \dd u,
\notag
\end{align}
where the variable transformation $t = \varphi(u)/\varphi(N_{0})$ is employed. 
Furthermore, by letting 
\begin{align}
c_{r} := \frac{\pi}{2 \sqrt{r}}
\qquad
\text{and}
\qquad 
\eta := \frac{1}{2c_{r}}\log (s\, \varphi(N_{0})), 
\label{eq:alt_numbers}
\end{align}
we have
\begin{align}
& g(s\, \varphi(N_{0}), \, \varphi(u)) 
= 
g(\exp(2c_{r} \eta), \, \exp(2c_{r} \sqrt{u})) \notag \\
& =
\log \left| \frac{ 1 + \exp[2c_{r} (\sqrt{u} - \eta)] }{ 1 - \exp[2c_{r} (\sqrt{u} - \eta)] } \right| 
=
- \log \left| \tanh[c_{r} (\sqrt{u} - \eta)] \right|.
\notag
\end{align}
Note that $\eta$ in \eqref{eq:alt_numbers} satisfies $-\infty < \eta \leq \sqrt{N_{0}}$ 
and does not equal the square root of any nonnegative integer because $s \in (0,1] \setminus \{ a_{k} \}$.
From these, by letting
\begin{align}
G_{\eta}(u) := - \log \left| \tanh[c_{r} (\sqrt{u} - \eta)] \right|, 
\label{eq:def_logtanh}
\end{align}
we have
\begin{align}
\label{eq:cont_vs_disc_trans}
& \sum_{k = 1}^{N_{0}+1} \log \left| \frac{s + a_{k}}{s - a_{k}} \right| 
-
\int_{1/\varphi(N_{0})}^{1} \log \left| \frac{s + t}{s - t} \right| \frac{2r \log(t\, \varphi(N_{0}))}{\pi^{2} t} \, \dd t \\
& = 
\sum_{\ell = 0}^{N_{0}-1} G_{\eta}(\ell) + G_{\eta}(N_{0} - 1/2) 
- 
\int_{0}^{N_{0}} G_{\eta}(u)\, \dd u. \notag
\end{align}
In the following, 
we prove the RHS of \eqref{eq:cont_vs_disc_trans} is bounded from below for $\eta$ with $-\infty < \eta \leq \sqrt{N_{0}}$,
which completes the proof of Lemma~\ref{lem:cont_vs_disc}. 
Note that, as shown by Figure~\ref{fig:graphs_of_G}, 
the function $G_{\eta}(u)$ in \eqref{eq:def_logtanh} is nonnegative and has a singularity at 
\begin{align}
u = \eta^{2} = \frac{r}{\pi^{2}} [\log (s\, \varphi(N_{0}))]^{2}
\label{eq:singularity}
\end{align}
in the case that $\eta > 0$. 

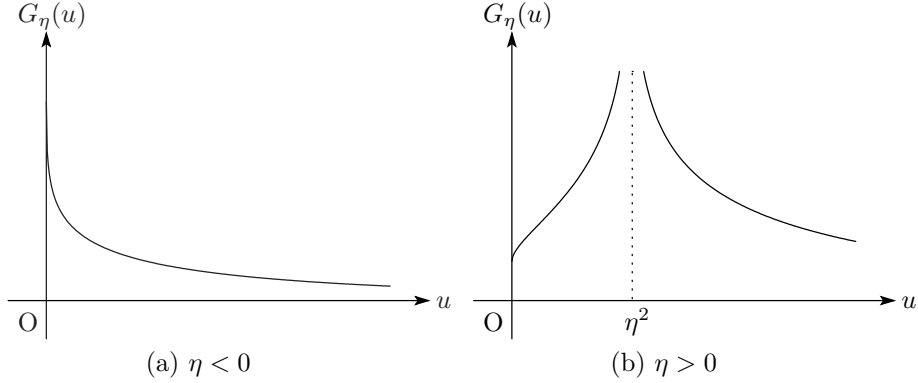
\begin{figure}[t]
\centering
\begin{minipage}{0.48\linewidth}
\centering
{\unitlength 0.1in%
\begin{picture}(24.4000,17.5000)(4.0000,-20.0000)%
\put(28.4000,-18.4000){\makebox(0,0)[lb]{$u$}}%
\put(6.5000,-19.6000){\makebox(0,0)[lb]{O}}%
\put(6.5000,-3.8000){\makebox(0,0)[lb]{$G_{\eta}(u)$}}%
\special{pn 8}%
\special{pa 800 759}%
\special{pa 805 1019}%
\special{pa 810 1087}%
\special{pa 815 1131}%
\special{pa 820 1164}%
\special{pa 825 1191}%
\special{pa 830 1213}%
\special{pa 835 1233}%
\special{pa 840 1250}%
\special{pa 845 1265}%
\special{pa 850 1279}%
\special{pa 860 1303}%
\special{pa 870 1323}%
\special{pa 880 1341}%
\special{pa 890 1357}%
\special{pa 910 1385}%
\special{pa 915 1391}%
\special{pa 920 1396}%
\special{pa 925 1402}%
\special{pa 930 1407}%
\special{pa 935 1413}%
\special{pa 940 1418}%
\special{pa 945 1422}%
\special{pa 955 1432}%
\special{pa 990 1460}%
\special{pa 995 1463}%
\special{pa 1000 1467}%
\special{pa 1015 1476}%
\special{pa 1020 1480}%
\special{pa 1035 1489}%
\special{pa 1040 1491}%
\special{pa 1055 1500}%
\special{pa 1060 1502}%
\special{pa 1065 1505}%
\special{pa 1070 1507}%
\special{pa 1075 1510}%
\special{pa 1080 1512}%
\special{pa 1085 1515}%
\special{pa 1095 1519}%
\special{pa 1100 1522}%
\special{pa 1165 1548}%
\special{pa 1170 1549}%
\special{pa 1185 1555}%
\special{pa 1190 1556}%
\special{pa 1200 1560}%
\special{pa 1205 1561}%
\special{pa 1210 1563}%
\special{pa 1215 1564}%
\special{pa 1220 1566}%
\special{pa 1225 1567}%
\special{pa 1230 1569}%
\special{pa 1235 1570}%
\special{pa 1240 1572}%
\special{pa 1245 1573}%
\special{pa 1250 1575}%
\special{pa 1255 1576}%
\special{pa 1260 1578}%
\special{pa 1270 1580}%
\special{pa 1275 1582}%
\special{pa 1285 1584}%
\special{pa 1290 1586}%
\special{pa 1305 1589}%
\special{pa 1310 1591}%
\special{pa 1330 1595}%
\special{pa 1335 1597}%
\special{pa 1370 1604}%
\special{pa 1375 1606}%
\special{pa 1440 1619}%
\special{pa 1445 1619}%
\special{pa 1485 1627}%
\special{pa 1490 1627}%
\special{pa 1515 1632}%
\special{pa 1520 1632}%
\special{pa 1540 1636}%
\special{pa 1545 1636}%
\special{pa 1560 1639}%
\special{pa 1565 1639}%
\special{pa 1580 1642}%
\special{pa 1585 1642}%
\special{pa 1600 1645}%
\special{pa 1605 1645}%
\special{pa 1615 1647}%
\special{pa 1620 1647}%
\special{pa 1630 1649}%
\special{pa 1635 1649}%
\special{pa 1645 1651}%
\special{pa 1650 1651}%
\special{pa 1660 1653}%
\special{pa 1665 1653}%
\special{pa 1675 1655}%
\special{pa 1680 1655}%
\special{pa 1690 1657}%
\special{pa 1695 1657}%
\special{pa 1700 1658}%
\special{pa 1705 1658}%
\special{pa 1715 1660}%
\special{pa 1720 1660}%
\special{pa 1725 1661}%
\special{pa 1730 1661}%
\special{pa 1735 1662}%
\special{pa 1740 1662}%
\special{pa 1750 1664}%
\special{pa 1755 1664}%
\special{pa 1760 1665}%
\special{pa 1765 1665}%
\special{pa 1770 1666}%
\special{pa 1775 1666}%
\special{pa 1785 1668}%
\special{pa 1790 1668}%
\special{pa 1795 1669}%
\special{pa 1800 1669}%
\special{pa 1805 1670}%
\special{pa 1810 1670}%
\special{pa 1815 1671}%
\special{pa 1820 1671}%
\special{pa 1825 1672}%
\special{pa 1830 1672}%
\special{pa 1835 1673}%
\special{pa 1840 1673}%
\special{pa 1845 1674}%
\special{pa 1850 1674}%
\special{pa 1855 1675}%
\special{pa 1860 1675}%
\special{pa 1865 1676}%
\special{pa 1870 1676}%
\special{pa 1875 1677}%
\special{pa 1880 1677}%
\special{pa 1885 1678}%
\special{pa 1890 1678}%
\special{pa 1895 1679}%
\special{pa 1900 1679}%
\special{pa 1905 1680}%
\special{pa 1915 1680}%
\special{pa 1920 1681}%
\special{pa 1925 1681}%
\special{pa 1930 1682}%
\special{pa 1935 1682}%
\special{pa 1940 1683}%
\special{pa 1945 1683}%
\special{pa 1950 1684}%
\special{pa 1960 1684}%
\special{pa 1965 1685}%
\special{pa 1970 1685}%
\special{pa 1975 1686}%
\special{pa 1980 1686}%
\special{pa 1985 1687}%
\special{pa 1995 1687}%
\special{pa 2000 1688}%
\special{pa 2005 1688}%
\special{pa 2010 1689}%
\special{pa 2020 1689}%
\special{pa 2025 1690}%
\special{pa 2030 1690}%
\special{pa 2035 1691}%
\special{pa 2045 1691}%
\special{pa 2050 1692}%
\special{pa 2055 1692}%
\special{pa 2060 1693}%
\special{pa 2070 1693}%
\special{pa 2075 1694}%
\special{pa 2085 1694}%
\special{pa 2090 1695}%
\special{pa 2095 1695}%
\special{pa 2100 1696}%
\special{pa 2110 1696}%
\special{pa 2115 1697}%
\special{pa 2125 1697}%
\special{pa 2130 1698}%
\special{pa 2140 1698}%
\special{pa 2145 1699}%
\special{pa 2150 1699}%
\special{pa 2155 1700}%
\special{pa 2165 1700}%
\special{pa 2170 1701}%
\special{pa 2180 1701}%
\special{pa 2185 1702}%
\special{pa 2195 1702}%
\special{pa 2200 1703}%
\special{pa 2210 1703}%
\special{pa 2215 1704}%
\special{pa 2230 1704}%
\special{pa 2235 1705}%
\special{pa 2245 1705}%
\special{pa 2250 1706}%
\special{pa 2260 1706}%
\special{pa 2265 1707}%
\special{pa 2275 1707}%
\special{pa 2280 1708}%
\special{pa 2290 1708}%
\special{pa 2295 1709}%
\special{pa 2310 1709}%
\special{pa 2315 1710}%
\special{pa 2325 1710}%
\special{pa 2330 1711}%
\special{pa 2345 1711}%
\special{pa 2350 1712}%
\special{pa 2360 1712}%
\special{pa 2365 1713}%
\special{pa 2380 1713}%
\special{pa 2385 1714}%
\special{pa 2400 1714}%
\special{pa 2405 1715}%
\special{pa 2415 1715}%
\special{pa 2420 1716}%
\special{pa 2435 1716}%
\special{pa 2440 1717}%
\special{pa 2455 1717}%
\special{pa 2460 1718}%
\special{pa 2475 1718}%
\special{pa 2480 1719}%
\special{pa 2495 1719}%
\special{pa 2500 1720}%
\special{pa 2515 1720}%
\special{pa 2520 1721}%
\special{pa 2535 1721}%
\special{pa 2540 1722}%
\special{pa 2555 1722}%
\special{pa 2560 1723}%
\special{pa 2580 1723}%
\special{pa 2585 1724}%
\special{pa 2600 1724}%
\special{fp}%
%
\special{pn 8}%
\special{pa 800 2000}%
\special{pa 800 400}%
\special{fp}%
\special{sh 1}%
\special{pa 800 400}%
\special{pa 780 467}%
\special{pa 800 453}%
\special{pa 820 467}%
\special{pa 800 400}%
\special{fp}%
%
\special{pn 8}%
\special{pa 600 1800}%
\special{pa 2800 1800}%
\special{fp}%
\special{sh 1}%
\special{pa 2800 1800}%
\special{pa 2733 1780}%
\special{pa 2747 1800}%
\special{pa 2733 1820}%
\special{pa 2800 1800}%
\special{fp}%
\end{picture}}%
 \\
(a) $\eta < 0$
\end{minipage}
\begin{minipage}{0.48\linewidth}
\centering
{\unitlength 0.1in%
\begin{picture}(24.4000,17.5000)(4.0000,-20.0000)%
\put(28.4000,-18.4000){\makebox(0,0)[lb]{$u$}}%
\put(6.5000,-19.6000){\makebox(0,0)[lb]{O}}%
\put(6.5000,-3.8000){\makebox(0,0)[lb]{$G_{\eta}(u)$}}%
\special{pn 8}%
\special{pa 800 1595}%
\special{pa 805 1569}%
\special{pa 810 1557}%
\special{pa 820 1539}%
\special{pa 840 1511}%
\special{pa 865 1481}%
\special{pa 870 1476}%
\special{pa 875 1470}%
\special{pa 880 1465}%
\special{pa 885 1459}%
\special{pa 895 1449}%
\special{pa 900 1443}%
\special{pa 920 1423}%
\special{pa 925 1417}%
\special{pa 945 1397}%
\special{pa 950 1391}%
\special{pa 965 1376}%
\special{pa 970 1370}%
\special{pa 985 1355}%
\special{pa 990 1349}%
\special{pa 1000 1339}%
\special{pa 1005 1333}%
\special{pa 1010 1328}%
\special{pa 1015 1322}%
\special{pa 1020 1317}%
\special{pa 1025 1311}%
\special{pa 1030 1306}%
\special{pa 1040 1294}%
\special{pa 1045 1289}%
\special{pa 1095 1229}%
\special{pa 1100 1222}%
\special{pa 1105 1216}%
\special{pa 1110 1209}%
\special{pa 1115 1203}%
\special{pa 1120 1196}%
\special{pa 1125 1190}%
\special{pa 1150 1155}%
\special{pa 1155 1147}%
\special{pa 1160 1140}%
\special{pa 1165 1132}%
\special{pa 1170 1125}%
\special{pa 1190 1093}%
\special{pa 1195 1084}%
\special{pa 1200 1076}%
\special{pa 1220 1040}%
\special{pa 1225 1030}%
\special{pa 1230 1021}%
\special{pa 1245 991}%
\special{pa 1260 958}%
\special{pa 1265 946}%
\special{pa 1270 935}%
\special{pa 1275 922}%
\special{pa 1280 910}%
\special{pa 1285 897}%
\special{pa 1300 855}%
\special{pa 1310 825}%
\special{pa 1325 774}%
\special{pa 1335 736}%
\special{pa 1340 715}%
\special{pa 1345 693}%
\special{pa 1350 670}%
\special{pa 1355 645}%
\special{pa 1360 619}%
\special{pa 1363 600}%
\special{fp}%
\special{pa 1490 600}%
\special{pa 1490 601}%
\special{pa 1495 626}%
\special{pa 1500 649}%
\special{pa 1505 671}%
\special{pa 1510 691}%
\special{pa 1520 729}%
\special{pa 1525 746}%
\special{pa 1535 778}%
\special{pa 1540 792}%
\special{pa 1545 807}%
\special{pa 1560 846}%
\special{pa 1565 858}%
\special{pa 1580 891}%
\special{pa 1585 901}%
\special{pa 1590 912}%
\special{pa 1595 921}%
\special{pa 1600 931}%
\special{pa 1610 949}%
\special{pa 1615 957}%
\special{pa 1620 966}%
\special{pa 1635 990}%
\special{pa 1640 997}%
\special{pa 1645 1005}%
\special{pa 1665 1033}%
\special{pa 1670 1039}%
\special{pa 1675 1046}%
\special{pa 1700 1076}%
\special{pa 1705 1081}%
\special{pa 1710 1087}%
\special{pa 1715 1092}%
\special{pa 1720 1098}%
\special{pa 1755 1133}%
\special{pa 1760 1137}%
\special{pa 1765 1142}%
\special{pa 1770 1146}%
\special{pa 1775 1151}%
\special{pa 1780 1155}%
\special{pa 1785 1160}%
\special{pa 1825 1192}%
\special{pa 1830 1195}%
\special{pa 1840 1203}%
\special{pa 1845 1206}%
\special{pa 1850 1210}%
\special{pa 1855 1213}%
\special{pa 1860 1217}%
\special{pa 1865 1220}%
\special{pa 1870 1224}%
\special{pa 1880 1230}%
\special{pa 1885 1234}%
\special{pa 1950 1273}%
\special{pa 1955 1275}%
\special{pa 1970 1284}%
\special{pa 1975 1286}%
\special{pa 1980 1289}%
\special{pa 1985 1291}%
\special{pa 1990 1294}%
\special{pa 1995 1296}%
\special{pa 2005 1302}%
\special{pa 2015 1306}%
\special{pa 2020 1309}%
\special{pa 2025 1311}%
\special{pa 2030 1314}%
\special{pa 2040 1318}%
\special{pa 2045 1321}%
\special{pa 2060 1327}%
\special{pa 2065 1330}%
\special{pa 2095 1342}%
\special{pa 2100 1345}%
\special{pa 2140 1361}%
\special{pa 2145 1362}%
\special{pa 2180 1376}%
\special{pa 2185 1377}%
\special{pa 2200 1383}%
\special{pa 2205 1384}%
\special{pa 2220 1390}%
\special{pa 2225 1391}%
\special{pa 2235 1395}%
\special{pa 2240 1396}%
\special{pa 2250 1400}%
\special{pa 2255 1401}%
\special{pa 2260 1403}%
\special{pa 2265 1404}%
\special{pa 2275 1408}%
\special{pa 2280 1409}%
\special{pa 2285 1411}%
\special{pa 2290 1412}%
\special{pa 2295 1414}%
\special{pa 2300 1415}%
\special{pa 2305 1417}%
\special{pa 2310 1418}%
\special{pa 2315 1420}%
\special{pa 2320 1421}%
\special{pa 2325 1423}%
\special{pa 2330 1424}%
\special{pa 2335 1426}%
\special{pa 2345 1428}%
\special{pa 2350 1430}%
\special{pa 2355 1431}%
\special{pa 2360 1433}%
\special{pa 2370 1435}%
\special{pa 2375 1437}%
\special{pa 2385 1439}%
\special{pa 2390 1441}%
\special{pa 2400 1443}%
\special{pa 2405 1445}%
\special{pa 2415 1447}%
\special{pa 2420 1449}%
\special{pa 2435 1452}%
\special{pa 2440 1454}%
\special{pa 2455 1457}%
\special{pa 2460 1459}%
\special{pa 2475 1462}%
\special{pa 2480 1464}%
\special{pa 2505 1469}%
\special{pa 2510 1471}%
\special{pa 2545 1478}%
\special{pa 2550 1480}%
\special{pa 2600 1490}%
\special{fp}%
%
\special{pn 8}%
\special{pa 1430 1790}%
\special{pa 1430 600}%
\special{dt 0.045}%
\special{pa 1430 600}%
\special{pa 1430 600}%
\special{dt 0.045}%
\put(13.9000,-19.8000){\makebox(0,0)[lb]{$\eta^{2}$}}%
%
\special{pn 8}%
\special{pa 800 2000}%
\special{pa 800 400}%
\special{fp}%
\special{sh 1}%
\special{pa 800 400}%
\special{pa 780 467}%
\special{pa 800 453}%
\special{pa 820 467}%
\special{pa 800 400}%
\special{fp}%
%
\special{pn 8}%
\special{pa 600 1800}%
\special{pa 2800 1800}%
\special{fp}%
\special{sh 1}%
\special{pa 2800 1800}%
\special{pa 2733 1780}%
\special{pa 2747 1800}%
\special{pa 2733 1820}%
\special{pa 2800 1800}%
\special{fp}%
\end{picture}}%
 \\
(b) $\eta > 0$
\end{minipage}
\caption{Graphs of the function $G_{\eta}$. }
\label{fig:graphs_of_G}
\end{figure}

The RHS of \eqref{eq:cont_vs_disc_trans} can be rewritten in the form
\begin{align}
\label{eq:cont_vs_disc_trans_rewritten}
& \frac{1}{2} G_{\eta}(0) 
+ \sum_{\ell = 0}^{N_{0} - 2} \left[ \frac{1}{2} G_{\eta}(\ell) + \frac{1}{2} G_{\eta}(\ell+1) - \int_{\ell}^{\ell + 1} G_{\eta}(u)\, \dd u \right] \\
& + \frac{1}{2} G_{\eta}(N_{0} - 1) 
+ \left[ G_{\eta}(N_{0} - 1/2) - \int_{N_{0} - 1}^{N_{0}} G_{\eta}(u)\, \dd u \right].
\notag 
\end{align}
The first and third terms in \eqref{eq:cont_vs_disc_trans_rewritten}
may be ignored because they are nonnegative. 
For the second and fourth terms in \eqref{eq:cont_vs_disc_trans_rewritten}, 
we divide the arguments for their estimates into the following three steps. 

\begin{description}
\item[Step 1] Estimate of some terms in the sum of the second term in \eqref{eq:cont_vs_disc_trans_rewritten} by the convexity of $G_{\eta}$.
Each term in the sum is the error of the trapezoidal approximation of the integral of $G_{\eta}$. 
The error on the interval $[\ell, \ell+1]$ is nonnegative if $G_{\eta}$ is convex on the interval. 
Therefore, we bound the error from below by zero on the interval where $G_{\eta}$ is convex. 
\item[Step 2] Estimate of the other errors in the second term on the intervals in which the singularity $u = \eta^{2}$ is not contained. 
\item[Step 3] Estimate of the error in the second term on the interval in which the singularity $u = \eta^{2}$ is contained if it exists and 
estimate of the fourth term in \eqref{eq:cont_vs_disc_trans_rewritten}. 
Note that this step is necessary only if $\eta > 0$. 
Therefore, we assume that $\eta > 0$ in this step.
Then, the singularity is contained in $(N_{0} - 1, N_{0})$ or $(0, N_{0}-1)$. 
In the former case, we have only to estimate the fourth term in this step 
because all the errors in the second term are already estimated in the previous step. 
\end{description}

\subsubsection{Step 1}

We show a sufficient condition for the convexity of $G_{\eta}$. 

\begin{lem}
\label{lem:convex}
If $u > r/\pi^{2}$ and $u \neq \eta^{2}$, then the function $G_{\eta}(u)$ given by \eqref{eq:def_logtanh} is convex. 
\end{lem}

\begin{proof}
Since
\begin{align}
G_{\eta}'(u) 
& =
- \frac{c_{r} u^{-1/2}}{\sinh[2c_{r} (\sqrt{u} - \eta)]}, 
\notag 
\end{align}
we have
\begin{align}
G_{\eta}''(u) 
& = \frac{c_{r}}{ 2 u^{3/2}}\, \frac{ \cosh[2c_{r} (\sqrt{u} - \eta)] }{ \sinh^{2}[2c_{r} (\sqrt{u} - \eta)] }
\left(
\tanh[2c_{r} (\sqrt{u} - \eta)] + 2c_{r} u^{1/2}
\right).
\end{align}
If $u > r/\pi^{2}$ and $u \neq \eta^{2}$, then $2c_{r} u^{1/2} > 1$ and $G_{\eta}''(u) > 0$.
\end{proof}

We define indexes $\ell_{\mathrm{c}}$ and $\ell_{\mathrm{s}}$ by 
\begin{align}
\ell_{\mathrm{c}} := \lceil r/\pi^{2} \rceil \qquad \text{and} \qquad \ell_{\mathrm{s}} := \lfloor \eta^{2} \rfloor,
\end{align} 
respectively. 
By Lemma~\ref{lem:convex}, we have
\begin{align}
\sum_{
\begin{subarray}{c}
\ell_{\mathrm{c}} \leq \ell \leq N_{0}-2, \\
\ell \neq \ell_{\mathrm{s}}
\end{subarray}} 
\left[ \frac{1}{2} G_{\eta}(\ell) + \frac{1}{2} G_{\eta}(\ell+1) - \int_{\ell}^{\ell + 1} G_{\eta}(u)\, \dd u \right]
\geq 0,
\end{align}
which is the desired inequality in Step 1.

\subsubsection{Step 2}

We start with the estimate 
\begin{align}
\sum_{
\begin{subarray}{c}
0 \leq \ell \leq \ell_{\mathrm{c}}, \\
\ell \neq \ell_{\mathrm{s}}
\end{subarray}} 
\left[ \frac{1}{2} G_{\eta}(\ell) + \frac{1}{2} G_{\eta}(\ell+1) - \int_{\ell}^{\ell + 1} G_{\eta}(u)\, \dd u \right]
\geq
\sum_{
\begin{subarray}{c}
0 \leq \ell \leq \ell_{\mathrm{c}}, \\
\ell \neq \ell_{\mathrm{s}}
\end{subarray}} 
\int_{\ell}^{\ell + 1} -G_{\eta}(u)\, \dd u. 
\label{eq:Step2first}
\end{align}
By the variable transformation $v = \sqrt{u}$, for $\ell$ with $0 \leq \ell \leq \ell_{\mathrm{c}}$, we have
\begin{align}
\int_{\ell}^{\ell + 1} -G_{\eta}(u)\, \dd u
& = 
2 \int_{\sqrt{\ell}}^{\sqrt{\ell + 1}} \log \left| \tanh [c_{r} (v - \eta)] \right| v \, \dd v \notag \\
& \geq 
2 \sqrt{\ell_{\mathrm{c}} + 1} \int_{\sqrt{\ell}}^{\sqrt{\ell + 1}} \log \left| \tanh [c_{r} (v - \eta)] \right| \, \dd v.
\notag
\end{align}
Then, we have
\begin{align}
\label{eq:intlogtanh}
& \sum_{
\begin{subarray}{c}
0 \leq \ell \leq \ell_{\mathrm{c}}, \\
\ell \neq \ell_{\mathrm{s}}
\end{subarray}} 
\int_{\ell}^{\ell + 1} -G_{\eta}(u)\, \dd u \\
& \geq
2 \sqrt{\ell_{\mathrm{c}} + 1} \int_{[0, \sqrt{\ell_{\mathrm{c}} + 1}] \setminus [\sqrt{\ell_{\mathrm{s}}}, \sqrt{\ell_{\mathrm{s}} + 1}]} 
\log \left| \tanh [c_{r} (v - \eta)] \right| \, \dd v \notag \\
& \geq
2 \sqrt{\ell_{\mathrm{c}} + 1} \int_{-\infty}^{\infty} 
\log \left| \tanh [c_{r} (v - \eta)] \right| \, \dd v \notag \\
& =
2 \sqrt{\ell_{\mathrm{c}} + 1} \left( - \frac{\pi^{2}}{4 c_{r}} \right)
= 
- \pi \sqrt{r} \sqrt{ \left\lceil \frac{r}{\pi^{2}} \right\rceil + 1 }. 
\notag
\end{align}
Inequalities~\eqref{eq:Step2first} and~\eqref{eq:intlogtanh} give the desired estimate in Step 2.

\subsubsection{Step 3}

We need to consider the following three cases: 
(i) $\eta^{2} \in (N_{0}-1, N_{0})$,  
(ii) $\eta^{2} \in (N_{0}-2, N_{0}-1)$, and 
(iii) $\eta^{2} \in (0, N_{0}-2)$. 
In Case (i), 
we have $\ell_{\mathrm{s}} = N_{0} - 1$ and 
have only to estimate the fourth term in \eqref{eq:cont_vs_disc_trans_rewritten}:
\begin{align}
G_{\eta}(N_{0} - 1/2) - \int_{N_{0} - 1}^{N_{0}} G_{\eta}(u)\, \dd u.
\label{eq:fourth_term}
\end{align}
In Cases (ii) and (iii), 
we have $0 \leq \ell_{\mathrm{s}} \leq N_{0} - 2$ and need to estimate the value in~\eqref{eq:fourth_term} and 
the error
\begin{align}
\frac{1}{2} G_{\eta}(\ell_{\mathrm{s}}) + \frac{1}{2} G_{\eta}(\ell_{\mathrm{s}}+1) 
- \int_{\ell_{\mathrm{s}}}^{\ell_{\mathrm{s}} + 1} G_{\eta}(u)\, \dd u. 
\label{eq:singular_error}
\end{align}

For these estimates, we need the following two lemmas. 

\begin{lem}
\label{lem:G_lower}
Let $m$ be an integer satisfying $m > 1$ and $m - 1 < \eta^{2} < m+1$, and 
let $u$ be a real number satisfying $m \leq u \leq m+1$. 
Then, we have
\begin{align}
G_{\eta}(u) \geq \frac{1}{2} \log (m-1) - \log c_{r}. 
\notag
\end{align}
\end{lem}
\begin{proof}
Noting that $| \tanh x | \leq |x|$ for any $x \in \mathbf{R}$, 
we have
\begin{align}
G_{\eta}(u)
& =
- \log \left| \tanh[c_{r} (\sqrt{u} - \eta)] \right|
\geq 
- \log \left| c_{r} (\sqrt{u} - \eta) \right|
= 
- \log \left| \frac{c_{r} (u - \eta^{2})}{\sqrt{u} + \eta} \right| \notag \\
& = 
\log \left| \sqrt{u} + \eta \right| - \log \left| c_{r} (u - \eta^{2}) \right| 
\geq
\log \left| \sqrt{m} + \sqrt{m-1} \right| - \log (2 c_{r} ) \notag \\
& \geq
\frac{1}{2} \log (m-1) - \log c_{r}.  
\notag
\end{align}
\end{proof}
\begin{lem}
\label{lem:int_minus_G_lower}
Let $m$ be a nonnegative integer satisfying $m - 1 < \eta^{2} < m+1$, and let $t_{r} = \tanh (2c_{r})/2$. 
Then, we have
\begin{align}
- \int_{m}^{m + 1} G_{\eta}(u)\, \dd u 
\geq 
- \frac{1}{2} \log (m+1) + \log \frac{t_{r}}{4} - \frac{3}{2}. 
\notag
\end{align}
\end{lem}
\begin{proof}
Because $| \tanh (c_{r} x) | \geq | (\tanh (2c_{r})/2) \, x |$ for any $x \in \mathbf{R}$ with $|x| \leq 2$, 
by letting $t_{r} = \tanh (2c_{r})/2$, 
we have 
\begin{align}
\label{eq:int_G_start_estim}
& - \int_{m}^{m + 1} G_{\eta}(u)\, \dd u 
= 
\int_{m}^{m + 1} \log \left| \tanh[ c_{r} (\sqrt{u} - \eta) ] \right| \, \dd u \\
& \geq 
\int_{m}^{m + 1} \log \left| t_{r} (\sqrt{u} - \eta) \right| \, \dd u 
\geq 
\int_{\sqrt{m}}^{\sqrt{m + 1}} 2 \log \left| v - \eta \right| \, v\, \dd v + \log t_{r}.
\notag 
\end{align}
For the estimate of the RHS in \eqref{eq:int_G_start_estim}, 
we use the following indefinite integral:
\begin{align}
\int 2\log \left| v - \eta \right| \, v\, \dd v
= 
(v^{2} - \eta^{2}) \log |v - \eta| - \frac{1}{2} v^{2} - \eta v + c',
\notag
\end{align}
where $c'$ is a constant independent of $v$. 
In the following, we consider the following two cases: 
(a) $m - 1 < \eta^{2} < m$ and
(b) $m < \eta^{2} < m + 1$.
Note that Case (a) is void if $m = 0$.  

\bigskip

\noindent
\underline{Case (a)}\\

We have
\begin{align}
& \int_{\sqrt{m}}^{\sqrt{m + 1}} 2 \log \left| v - \eta \right| \, v\, \dd v \notag \\
& = 
(m+1 - \eta^{2}) \log (\sqrt{m+1} - \eta) - \frac{1}{2}(m+1) - \eta \sqrt{m+1} \notag \\
& \phantom{=} \ -\left[ (m - \eta^{2}) \log (\sqrt{m} - \eta) - \frac{1}{2}m - \eta \sqrt{m} \right] \notag \\
& = 
\log (\sqrt{m+1} - \eta) + (m - \eta^{2}) \log \frac{\sqrt{m+1} - \eta}{\sqrt{m} - \eta}
- \frac{1}{2} - \frac{\eta}{\sqrt{m+1} + \sqrt{m}} \notag \\
& \geq 
\log (\sqrt{m+1} - \eta) - \frac{1}{2} - \frac{\eta}{\sqrt{m+1} + \sqrt{m}} \notag \\
& \geq 
\log (\sqrt{m+1} - \eta) - 1. \notag
\end{align}
Furthermore, the term $\log (\sqrt{m+1} - \eta)$ is bounded from below as follows:
\begin{align}
& \log (\sqrt{m+1} - \eta) 
\geq 
\log (\sqrt{m+1} - \sqrt{m}) \notag \\
& = 
\log \frac{1}{\sqrt{m+1} + \sqrt{m}}
\geq - \frac{1}{2} \log (m+1) + \log \frac{1}{2}.
\notag
\end{align}
From these and~\eqref{eq:int_G_start_estim}, we have
\begin{align}
- \int_{m}^{m + 1} G_{\eta}(u)\, \dd u 
\geq
- \frac{1}{2} \log (m+1) + \log \frac{t_{r}}{2} - 1.
\label{eq:int_G_Case_a}
\end{align}

\bigskip

\noindent
\underline{Case (b)}\\

We have 
\begin{align}
& \int_{\sqrt{m}}^{\sqrt{m + 1}} 2 \log \left| v - \eta \right| \, v\, \dd v 
= 
\left( \int_{\sqrt{m}}^{\eta} + \int_{\eta}^{\sqrt{m + 1}} \right) 2 \log \left| v - \eta \right| \, v\, \dd v 
\notag \\
& =
-\left[ (m - \eta^{2}) \log (\eta - \sqrt{m} ) - \frac{1}{2}m - \eta \sqrt{m} \right] \notag \\
& \phantom{=}\ + (m+1 - \eta^{2}) \log (\sqrt{m+1} - \eta) - \frac{1}{2}(m+1) - \eta \sqrt{m+1} \notag \\
& = 
(m+1 - \eta^{2}) \log (\sqrt{m+1} - \eta) + (\eta^{2} - m) \log (\eta - \sqrt{m} ) 
- \frac{1}{2} - \frac{\eta}{\sqrt{m+1} + \sqrt{m}} \notag \\
& \geq
(m+1 - \eta^{2}) \log (\sqrt{m+1} - \eta) + (\eta^{2} - m) \log (\eta - \sqrt{m} ) 
- \frac{3}{2}. 
\notag 
\end{align}
For the estimate of the first and second terms of the last line above, 
we employ Jensen's inequality 
\begin{align}
\lambda \, f(x_{1}) + (1- \lambda)\, f(x_{2}) \geq f(\lambda x_{1} + (1- \lambda)\, x_{2})
\qquad
(0 \leq \lambda \leq 1, \ x_{1}, x_{2}> 0)
\notag
\end{align}
for the function $f(x) = x \log x$ by letting 
\begin{align}
\lambda = \frac{\sqrt{m+1} + \eta}{\sqrt{m+1} + \sqrt{m} + 2\eta}, \quad 
x_{1} = \sqrt{m+1} - \eta, \quad
x_{2} = \eta - \sqrt{m}. 
\notag 
\end{align}
Then, we have 
\begin{align}
& (m+1 - \eta^{2}) \log (\sqrt{m+1} - \eta) + (\eta^{2} - m) \log (\eta - \sqrt{m} ) \notag \\
& \geq 
\log \frac{1}{\sqrt{m+1} + \sqrt{m} + 2\eta}
\geq 
\log \frac{1}{4 \sqrt{m+1}}
= -\frac{1}{2} \log (m + 1) + \log \frac{1}{4}. 
\notag
\end{align}
From these and~\eqref{eq:int_G_start_estim}, we have
\begin{align}
- \int_{m}^{m + 1} G_{\eta}(u)\, \dd u 
\geq
- \frac{1}{2} \log (m+1) + \log \frac{t_{r}}{4} - \frac{3}{2}.
\label{eq:int_G_Case_b}
\end{align}

Combining~\eqref{eq:int_G_Case_a} in Case (a) and~\eqref{eq:int_G_Case_b} in Case (b), 
we obtain the conclusion of Lemma~\ref{lem:int_minus_G_lower}.
\end{proof}

\bigskip

By using Lemmas~\ref{lem:G_lower} and~\ref{lem:int_minus_G_lower}, we finish Step 3. 

\bigskip

\noindent
\underline{Case (i)}\\

In this case, $\ell_{\mathrm{s}} = N_{0} - 1$ holds. 
By 
letting $m = N_{0} - 1$ in Lemmas~\ref{lem:G_lower} and~\ref{lem:int_minus_G_lower}, and
letting $u = N_{0} - 1/2$ in Lemma~\ref{lem:G_lower}, 
we have
\begin{align}
G_{\eta}(N_{0} - 1/2) - \int_{N_{0} - 1}^{N_{0}} G_{\eta}(u)\, \dd u
\geq 
\frac{1}{2} \log \frac{N_{0}-2}{N_{0}} + \log \frac{t_{r}}{4c_{r}} - \frac{3}{2}.
\label{eq:Case1_fourth_term}
\end{align}
Thus, the value in~\eqref{eq:fourth_term} is bounded from below by a constant independent of $N_{0}$ and $\eta$. 

\bigskip

\noindent
\underline{Case (ii)} \\

In this case, $\ell_{\mathrm{s}} = N_{0} - 2$ holds. 
First, 
we can estimate the value in~\eqref{eq:fourth_term} in the same manner as~\eqref{eq:Case1_fourth_term} in Case (i). 
Next, 
we estimate the value in \eqref{eq:singular_error}. 
By 
letting $m = \ell_{\mathrm{s}}$ in Lemmas~\ref{lem:G_lower} and~\ref{lem:int_minus_G_lower}, and 
letting $u = \ell_{\mathrm{s}}$ and $u = \ell_{\mathrm{s}}+1$ in Lemma~\ref{lem:G_lower}, 
we have
\begin{align}
\frac{1}{2} G_{\eta}(\ell_{\mathrm{s}}) + \frac{1}{2} G_{\eta}(\ell_{\mathrm{s}}+1) 
- \int_{\ell_{\mathrm{s}}}^{\ell_{\mathrm{s}} + 1} G_{\eta}(u)\, \dd u
\geq
\frac{1}{2} \log \frac{ \ell_{\mathrm{s}} - 1 }{\ell_{\mathrm{s}} + 1} 
+ \log \frac{t_{r}}{4c_{r}} - \frac{3}{2}.
\label{eq:Case2_trap_err}
\end{align}
Thus, the value in \eqref{eq:singular_error} is bounded from below by a constant independent of $N_{0}$ and $\eta$. 

\bigskip

\noindent
\underline{Case (iii)} \\

In this case, $0 \leq \ell_{\mathrm{s}} \leq N_{0} - 3$ holds. 
First, 
we estimate the value in \eqref{eq:singular_error}. 
In the case that $\ell_{\mathrm{s}} > 1$, 
we can estimate it in the same manner as~\eqref{eq:Case2_trap_err} in Case (ii).
In the case that $0 \leq \ell_{\mathrm{s}} \leq 1$,  
by Lemma~\ref{lem:int_minus_G_lower} with $m = \ell_{\mathrm{s}}$, 
we have
\begin{align}
& \frac{1}{2} G_{\eta}(\ell_{\mathrm{s}}) + \frac{1}{2} G_{\eta}(\ell_{\mathrm{s}}+1) 
- \int_{\ell_{\mathrm{s}}}^{\ell_{\mathrm{s}} + 1} G_{\eta}(u)\, \dd u \notag \\
& \geq
- \int_{\ell_{\mathrm{s}}}^{\ell_{\mathrm{s}} + 1} G_{\eta}(u)\, \dd u
\geq 
- \frac{1}{2} \log 2 + \log \frac{t_{r}}{4} - \frac{3}{2}. 
\notag
\end{align}
Thus, the value in \eqref{eq:singular_error} is bounded from below by a constant independent of $N_{0}$ and $\eta$. 
Next, 
we estimate the value in \eqref{eq:fourth_term}.
Because $G_{\eta}$ is monotone decreasing on $(\eta^{2}, \infty)$, we have
\begin{align}
\label{eq:fourth_term_final_estimate}
& G_{\eta}(N_{0} - 1/2) - \int_{N_{0} - 1}^{N_{0}} G_{\eta}(u)\, \dd u
\geq
G_{\eta}(N_{0} - 1/2) - G_{\eta}(N_{0} - 1) \\
& = 
- \log \left| \tanh[c_{r} (\sqrt{N_{0} - 1/2} - \eta)] \right|
+ \log \left| \tanh[c_{r} (\sqrt{N_{0} - 1} - \eta)] \right|. 
\notag 
\end{align}
If we regard the function given by the RHS in \eqref{eq:fourth_term_final_estimate} as a function of $\eta$, 
it is monotone decreasing because its derivative with respect to $\eta$ satisfies
\begin{align}
\frac{c_{r}}{\sinh[2c_{r} (\sqrt{N_{0}-1/2} - \eta)]} - 
\frac{c_{r}}{\sinh[2c_{r} (\sqrt{N_{0}-1} - \eta)]} < 0. 
\notag
\end{align}
Therefore, we have
\begin{align}
& - \log \left| \tanh[c_{r} (\sqrt{N_{0} - 1/2} - \eta)] \right|
+ \log \left| \tanh[c_{r} (\sqrt{N_{0} - 1} - \eta)] \right| \notag \\
& \geq
- \log \left| \tanh[c_{r} (\sqrt{N_{0} - 1/2} - \sqrt{N_{0}-2})] \right|
+ \log \left| \tanh[c_{r} (\sqrt{N_{0} - 1} - \sqrt{N_{0}-2})] \right| \notag \\
& \geq
- \log \left| c_{r} (\sqrt{N_{0} - 1/2} - \sqrt{N_{0}-2}) \right|
+ \log \left| (\tanh c_{r}) (\sqrt{N_{0} - 1} - \sqrt{N_{0}-2}) \right| \notag \\
& =
\log \left| \frac{\sqrt{N_{0} - 1} - \sqrt{N_{0}-2}}{\sqrt{N_{0} - 1/2} - \sqrt{N_{0}-2}} \right|
+ \log \frac{\tanh c_{r}}{ c_{r} } \notag \\
& = 
\log \left| \frac{\sqrt{N_{0} - 1/2} + \sqrt{N_{0}-2}}{\sqrt{N_{0} - 1} + \sqrt{N_{0}-2}} \right|
+ \log \frac{2 \tanh c_{r}}{ 3 c_{r} }. 
\notag
\end{align}
Thus, the value in~\eqref{eq:fourth_term} is bounded from below by a constant independent of $N_{0}$ and $\eta$. 

\bigskip

From the estimates in the three cases above, Step 3 is completed.
Thus, Lemma~\ref{lem:cont_vs_disc} is proven.

\section*{Acknowledgment}

T. Okayama and M. Sugihara are supported by the grant-in-aid of Japan Society for the Promotion of
Science (KAKENHI Grant Numbers: JP24760060 (Okayama), JP25390146 (Sugihara)).



\end{document}